\documentclass[11pt]{amsart}
\usepackage{}
\usepackage{a4wide}
\usepackage{mathrsfs}
\usepackage{amsfonts}
\usepackage{amsmath,extpfeil}
\usepackage{stmaryrd}
\usepackage{amssymb}
\usepackage{amsthm}
\usepackage{mathrsfs}
\usepackage{url}
\usepackage{amsfonts}
\usepackage{amscd}
\usepackage{indentfirst}
\usepackage{enumerate}
\usepackage{amsmath,amsfonts,amssymb,amsthm}
\usepackage{amsmath,amssymb,amsthm,amscd}
\usepackage{graphicx,mathrsfs}
\usepackage{appendix}
\usepackage[numbers,sort&compress]{natbib}
\usepackage{color}
\usepackage[colorlinks, linkcolor=blue, citecolor=blue]{hyperref}

\newcommand{\R}{\mathbb{R}}

\newcommand{\I}{\mathbb I}

\newtheorem{theorem}{Theorem}[section]
\newtheorem{lemma}{Lemma}[section]
\newtheorem{corollary}{Corollary}[section]
\newtheorem{example}{Example}[section]

\newtheorem{definition}{Definition}[section]% Use {\rm ...}
\newtheorem{remark}{Remark}[section]    % Use {\rm ...}
\numberwithin{equation}{section}

\newcommand\vep{\eps }

\newcommand\eps{\varepsilon}
\newcommand{\p}{\partial}
\newcommand{\dist}{\text{dist}}

\newcommand{\beq}{\begin{equation}}
\newcommand{\eeq}{\end{equation}}

\begin{document}

\title[Global $C^{1,\alpha}$ regularity]
{Global $C^{1,\alpha}$ regularity for Monge-Amp\`ere equations on planar convex domains}

\author[Q. Han]{Qing Han}
\address [Qing Han]{Department of Mathematics, University of Notre Dame, Notre Dame, IN 46556, USA}
\email{qhan@nd.edu}

\author[J. Liu]
{Jiakun Liu}
\address [Jiakun Liu]
	{School of Mathematics and Statistics,
	The University of Sydney,
	Camperdown, NSW 2006, AUSTRALIA}
\email{jiakun.liu@sydney.edu.au}

\author[Y. Zhou]{Yang Zhou}
\address [Yang Zhou]{The Institute of Mathematical Sciences, The Chinese University of Hong Kong, Shatin, NT, Hong Kong}
\email{yzhou@ims.cuhk.edu.hk}

\thanks{Q. Han acknowledges the support of NSF Grant DMS-2305038; J. Liu acknowledges the support of ARC FT220100368 and DP230100499; and Y. Zhou acknowledges the support of GS Charity Foundation Limited-Fundamental Mathematics Talent Development of IMS}



\begin{abstract}
In this paper, we establish the global H\"older gradient estimate for solutions to the Dirichlet problem of the Monge-Amp\`ere equation $\det D^2u = f$ on strictly convex but not uniformly convex domain $\Omega$.  
\end{abstract}

\date{\today}
\maketitle

\baselineskip=16.4pt
\parskip=3pt

%\tableofcontents

\section{Introduction} 

We consider the following Dirichlet problem for the Monge-Amp\`ere equation:
	\begin{equation}\label{u-fp}
	\begin{split}
		\det D^2u &=f \quad\text{in} \ \ \Omega,\\ 
		u & =\varphi \quad\text{on} \  \p\Omega,
	\end{split}
	\end{equation}
where $\Omega\subset\R^n$ is a bounded convex domain, $f\geq0$ is a measurable function in $\Omega$, and $\varphi$ is a given function on $\p\Omega$.

Caffarelli established the interior $C^{1,\alpha}$ regularity for strictly convex solution of \eqref{u-fp} under the assumption that $f$ satisfies a doubling condition, that is, there exists a constant $C_b > 0$ such that for any convex subset $D \subset \Omega$,
	\begin{equation}\label{dc}
		\int_{D}f (x) dx \le  C_b \int_{\frac12D}f (x) dx,
	\end{equation}
where $\frac12 D$ denotes the dilation of $D$ by a factor of $\frac12$ with respect to its mass centre (see \cite{C-1991}).
This implies $u\in C^{1,\alpha}(\Omega')$ for any $\Omega'\Subset\Omega$, with $\alpha\in(0,1)$ and the $C^{1,\alpha}$ norm depending on the Lipschitz norm and strict convexity of $u$.

Global $C^{1,\alpha}$ regularity for the Dirichlet problem \eqref{u-fp} has become a topic of great interest due to its fundamental importance in the study of PDEs with measurable coefficients, as well as its applications in optimisation, stochastic control, and computational geometry (see, e.g., \cite{GT,CTW-2022} and references therein). 

Under some restrictive conditions including $\varphi, \p\Omega \in C^{2,1}$ and $f^{1/n}\in C^{1,1}$, the global $C^{1,\alpha}$ regularity was obtained in \cite{B-1999} when $\Omega$ is uniformly convex. 
Assuming $\lambda\leq f\leq \Lambda$ for two constants $\Lambda\geq\lambda>0$, the boundary $C^{1,\alpha}$ regularity was obtained in \cite{SZ-2020} under an additional quadratic separation condition: If $\{x_{n+1}=0\}$ is the tangent plane of $u$ at $x_0\in\p\Omega$, then
	\begin{equation}\label{u-sq}
		\varphi(x) \ge c_0 |x-x_0|^2 \quad \text{ for any }\,x\in \p\Omega \text{ near }x_0,
	\end{equation}
where $c_0>0$ is a constant. This includes the case that $\Omega$ is uniformly convex, and $\p\Omega, \varphi \in C^{2,\beta}$ for some $\beta\in(0,1)$.
Recently, without condition \eqref{u-sq}, the global $C^{1,\alpha}$ regularity for \eqref{u-fp} was derived by Caffarelli, Tang and Wang \cite{CTW-2022} assuming that $f$ satisfies the doubling condition \eqref{dc} and allowing the degenerate case $f\geq0$. 

It is worth pointing out that $\Omega$ is also assumed to be uniformly convex in \cite{CTW-2022}. 
However, practical applications, such as image recognition and computational geometry often involve domains $\Omega$ that are convex but not uniformly convex (see \cite{PS-book}).
This raises the question whether the global $C^{1,\alpha}$ regularity can be achieved for \eqref{u-fp} in non-uniformly convex domains and under what minimal conditions on $\varphi$ and $f$.

In this paper, we investigate these questions by assuming $\Omega\subset\R^2$ and focusing on points of degeneracy at the boundary, where the curvature vanishes.
By a change of coordinates, we assume that the curvature of $\p\Omega$ vanishes at the origin and locally $\Omega$ is given by
	\begin{equation}\label{x2-rho-1}
		\Omega = \{x=(x_1,x_2) : x_2 > \rho(x_1)\}
	\end{equation}
for a convex function $\rho\geq0$ satisfying
	\begin{equation}\label{x2-rho-2}
		\rho(0)=0, \quad \rho'(0)=0,\quad\text{and} \quad \rho''(0)\geq0.
	\end{equation}
	
For the boundary $C^{1,\alpha}$ estimate, in addition to the doubling condition \eqref{dc} we also assume that $f$ is bounded from above, namely 
	\begin{equation}\label{f0}
		0 \le f(x) \le f_0 \ \ \text{ for $x$ near the boundary $\p\Omega$},
	\end{equation}
where  $f_0$ is a constant.

To illustrate challenges posed by degenerate boundaries, consider the following examples.

\begin{example}\label{eg1.1}
\emph{
Let $\rho(x_1)=x_1^4$ and $u(x)=x_2^{1+\alpha}$ with $0<\alpha<1$. 
Here, $u$ is a convex solution of the homogeneous equation $\det D^2u=0$. 
The boundary data $\varphi$ belongs to the class $C^{4+4\alpha}(\p\Omega)$, and the solution $u$ is in $C^{1,\alpha}(\overline\Omega)\cap C^\infty(\overline\Omega\setminus\{0\})$. Analogously, we can also consider the convex solution $x_1^{2k}+x_2^{1+\alpha}$ of the non-homogeneous equation. 
This example shows that $\varphi$ requires higher regularity (at least $C^{4,\beta}$) to ensure $u\in C^{1,\alpha}$. 
}
\end{example}

The above example suggests that higher the regularity of $\varphi$ is necessary, depending on the degeneracy of $\p\Omega$. 
We confirm this with a more detailed example from \cite{CTW-2022}. 
\begin{example}
\emph{
Consider a convex domain $\Omega$ in $\{0<x_2<1\}$,
symmetric in $x_1$, with $\p\Omega$ locally given by $\rho(x_1)=x_1^4$. 
For
	\begin{equation*}
		w(x)=-\frac{x_2}{\log x_2} + \delta x_1^2 x_2^\sigma
	\end{equation*}
for $\delta>0$ small and $\sigma\in(1/2,1)$ close to $1/2$, 
we have
	\begin{equation*}
		\det D^2w \geq \frac{\delta x_2^{\sigma-1}}{\log^2x_2}
	\end{equation*}
for $(x_1,x_2)\in\Omega$. 
Let $u$ be the solution to \eqref{u-fp} with $f$ being a small positive constant and $u=w$ on $\p\Omega$.
Then $u(0)=0$, and $u$ is symmetric in $x_1$ due to the symmetry of $\Omega$. By the comparison principle and the convexity of $u$, we get
	\begin{equation*}
		w(x_1,x_2) \leq u(x_1,x_2) \leq w(x_2^{1/4},x_2) \quad \text{ for any }\, (x_1,x_2)\in\Omega.	
	\end{equation*}
Hence, $Du(0)=0$ and $u(0,x_2) \geq x_2/|\log x_2|$. Note that the function $w\in C^4$ on $\p\Omega$ but $u$ is not $C^{1,\alpha}$ for any $\alpha>0$.
}
\end{example}
The above example indicates, again, that for the global $C^{1,\alpha}$ regularity, the boundary data needs a higher regularity depending on the degeneracy of convexity of $\p\Omega$. 
One can construct a similar example showing that if $\rho(x_1)=x_1^{2k}$ near $0\in\p\Omega$, in order for $u$ to be $C^{1,\alpha}$ regular, the boundary data $\varphi$ needs, at least, to be $C^{2k,\beta}$ for some $\beta>0$.  
It is worth noting that the degenerate case is significantly more complex than that of uniformly convex domains. See Remark \ref{rmkB1.1} below.

In order to state our main result, we introduce the following notation and definitions to describe the degeneracy of convexity of the boundary $\p\Omega$.

\begin{definition}\label{defdeg}
Let $k\geq4$ be an even integer. We say $\p\Omega$ 
is {\em $k$-order degenerate} at $0$ if
	\begin{equation}\label{k-deg}
		\rho''(0)=0, \ \ \rho'''(0)=0, \cdots, \rho^{(k-1)}(0)=0  \quad\text{but}\quad \rho^{(k)}(0)\ne 0,
	\end{equation}
where $\rho^{(i)}(x_1) = \frac{\p^i \rho(x_1)}{\p x_1^i}$, $i \in \mathbb N$.
\end{definition}

Let $\beta \in (0,1]$ and $\p\Omega\in C^{k+\beta}$. 
If $\p\Omega$ is $k$-order degenerate at $0$, we have 
	\begin{equation}\label{k+b-deg}
		\bar \rho(x_1) := \rho(x_1) - \frac{1}{k!} \rho^{(k)}(0) x_1^k \leq O\big(|x_1|^{k+\beta}\big).
	\end{equation} 
We assume that $\varphi\in C^{k,\beta}(\p\Omega)$ satisfies the following degeneracy condition:
\begin{itemize}
  \item[$(\mathcal P1)$] $\frac{\p^i \varphi}{\p x_1^i}(0) =0$\quad for $i=2,4, \cdots, k-2$.
\end{itemize}
Here, $\varphi$ is viewed as a function of $x_1$. 

Our first main theorem can be stated as below.

\begin{theorem}\label{mt-1}
Let $u$ be a convex solution of \eqref{u-fp} with $f$ satisfying \eqref{dc} and \eqref{f0}.
Assume $\Omega$ is bounded and convex, $\p\Omega\in C^{k+\beta}$ is $k$-order degenerate at $0$, and $\varphi\in C^{k,\beta}(\p\Omega)$ satisfies $(\mathcal P1)$.
Then, the solution $u$ is $C^{1,\alpha}(\overline\Omega)$ for some $\alpha \in (0,\beta/k)$, and 
	\begin{equation*}
		\|u\|_{C^{1,\alpha}(\overline\Omega)} \le C,
	\end{equation*}
where $\alpha, C$ are positive constants depending only on $C_b, f_0, k, \beta,\Omega$ and $\|\varphi\|_{C^{k,\beta}(\p\Omega)}$.
\end{theorem}

The following modification of Example \ref{eg1.1} shows that the condition $(\mathcal P1)$ is indeed necessary. 
\begin{example} 
\emph{
Let $k=4$, $\rho(x_1)=x_1^4$, and $u(x)=-x_2^{1/2}$. Then, $u$ is a convex solution of the homogeneous equation $\det D^2u=0$. 
The boundary data $\varphi=u|_{\partial\Omega}=-x_1^2$ is smooth but does not satisfy the condition $(\mathcal P1)$, since $\varphi''(0)=-2\neq0$. The solution $u$ is not $C^{0,1}$ at $0$. 
}
\end{example}

Note that Theorem \ref{mt-1} accommodates the degenerate case $f\geq0$. This naturally leads us to consider
the homogeneous Monge-Amp\`ere equation:
	\begin{equation}\label{u-f=0}
	\begin{split}
		\det D^2u &=0 \quad \, \text{in} \ \ \Omega,\\  
		u & = \varphi \quad\text{on} \ \ \p\Omega.
	\end{split}
	\end{equation}
For this case, we extend Definition \ref{defdeg} to accommodate $\beta\in(0,k]$. 
Specifically, we say that $\p\Omega$ is $(k,\beta)$-order degenerate at $0$ if both \eqref{k-deg} and \eqref{k+b-deg} hold for $\beta\in(0,k]$, namely
	\begin{equation}\label{k+b-deg1}
		\rho(x_1) = \frac{1}{k!} \rho^{(k)}(0) x_1^k + \bar \rho(x_1) \quad\text{ with }\  \bar \rho(x_1) \leq O\big(|x_1|^{k+\beta}\big).
	\end{equation}
In particular, when $\beta>1$, we assume that the boundary function $\varphi\in C^{k+\beta}(\p\Omega)$ also satisfies
\begin{itemize}
  \item[$(\mathcal P2)$] $\frac{\p^i \varphi}{\p x_1^i}(0) =0$\quad for $i=k+1, k+2, \cdots, k+[\beta]$,
\end{itemize}
where $[\beta]$ denotes the largest integer strictly less than $\beta$. 
Note that $[\beta]=\beta-1$ if $\beta$ is an integer. 

\begin{theorem}\label{mt-2}
Let $u$ be a convex solution to \eqref{u-f=0}.
Suppose that $\p\Omega\in C^{k+\beta}$, $\beta\in(0,k]$, is $(k,\beta)$-order degenerate at $0$, and that $\varphi\in C^{k+\beta}(\p\Omega)$ satisfies conditions $(\mathcal P1)$ and $(\mathcal P2)$.
Then, $u$ is $C^{1,\beta/k}(\overline\Omega)$ and  
	\begin{equation}\label{dges}
		\|u\|_{C^{1,\beta/k}(\overline\Omega)} \le C,
	\end{equation}
where $C$ is a positive constant depending only on $k, \beta,\Omega$ and $\|\varphi\|_{C^{k+\beta}(\p\Omega)}$.
\end{theorem}

We adopt the notation $C^{k+\beta}=C^{k,\beta}$ when $\beta\in(0,1)$, as in \cite{GT}.
When $\beta\geq1$, we write $C^{k+\beta}=C^{k+[\beta], \beta-[\beta]}$ with $\beta-[\beta]\in(0,1]$. 
For example, $C^{k+\beta}=C^{k,1}$ if $\beta=1$; and $C^{k+\beta}=C^{k+1,1}$ if $\beta=2$. 
In particular, when $\beta=k$, Theorem \ref{mt-2} yields a global $C^{1,1}$ regularity for \eqref{u-f=0}.
The global $C^{1,1}$ regularity for \eqref{u-f=0} was previously obtained in \cite{CNS-1986, GTW-1999,CTW-2022} under the assumption that $\p\Omega$ is uniformly convex. 
To our best knowledge, Theorem \ref{mt-2} is the first result on the global $C^{1,1}$ regularity for non-uniformly convex domains. 

The following example shows that if condition $(\mathcal P2)$ is not satisfied, the solution may fail to be in class $C^{1,\beta/k}(\overline\Omega)$.
\begin{example}
\emph{
Let $k=4$ and $\beta=2+\varepsilon$ for some $\vep>0$. 
Let $\rho(x_1)=x_1^4$, and $u=x_2+x_2^{3/2}$.
Then, $\det D^2u=0$, and $\varphi=x_1^4+x_1^6$ is smooth but does not satisfy $(\mathcal P2)$ since $\varphi^{(6)}(0)\neq0$.
We observe that $u$ is not $C^{1,\beta/k}$. 
}
\end{example}

A related problem is the regularity of the convex envelope, which appears in various analytical and geometrical contexts. 
Given a continuous function $w:\overline\Omega\to\mathbb{R}$, its convex envelope is defined by
	\begin{equation}\label{u-ce}
		u(x) = \sup \{ \ell(x) ~|~ \ell \le w \text{~in~} \overline\Omega \text{~and~} \ell \text{~is~affine}\}.
	\end{equation}
Global $C^{1,1}$ regularity for this setting was recently proved in \cite{DF-2015,CTW-2022} under the condition that $\partial\Omega$ is uniformly convex. 
For further background, see \cite{DF-2015,CTW-2022} and references therein. 
In this paper, we extend the global $C^{1,\alpha}$ estimate in Theorem \ref{mt-2} to solutions of \eqref{u-ce}, given that $w$ satisfies a corresponding condition at the degenerate point: 
\begin{itemize}
  \item[$(\mathcal P3)$] $D^{(i)}w(0) =0$\quad for $i=2, \cdots, k+[\beta]$.
\end{itemize}
Specifically, when $\beta=k$, we obtain the global $C^{1,1}$ regularity for non-uniformly convex domains.   
\begin{theorem}\label{mt-3}
Assume that $\p\Omega\in C^{k+\beta}$, $\beta\in(0,k]$, is $(k,\beta)$-order degenerate at point $0$.
Let $w\in C^{k+\beta}(\overline\Omega)$ satisfy $(\mathcal P3)$.
Then, the convex envelope $u$ defined by \eqref{u-ce} is in $C^{1,\beta/k}(\overline\Omega)$ and estimate \eqref{dges} holds. 
\end{theorem}

\begin{remark}\label{rmkB1.1}
\emph{
In this paper, we consider the non-uniformly convex domain $\Omega$, namely the degenerate order $k\geq4$. 
Handling the degeneracy of $\p\Omega$ is significantly more complex than the uniform convexity. For uniformly convex domains, each boundary point can be treated equally; however, when a domain is not uniformly convex -- for instance, $\Omega=\{x_1^4+x_2^2<1\}$ -- it is necessary to distinguish boundary points $x\in\p\Omega$ in two cases: $(i)$ if $x$ is near $(\pm1,0)$, $\p\Omega$ is locally quartic; $(ii)$ if $x$ is away from $(\pm1,0)$, $\p\Omega$ is locally quadratic. } 
\end{remark}

\begin{remark}
\emph{
Although we assume in this paper that the dimension $n=2$, our method and techniques may also be applicable in higher dimensions, particularly when $\p\Omega$ is isotropic at the degenerate point, for instance, $x_n=|x'|^k$, where $x'=(x_1,\cdots,x_{n-1})$. 
However, if $\p\Omega$ is anisotropic -- that is, the degenerate order varies along different directions -- the assumptions on $f$ and $\varphi$ become considerably more complex. 
We leave this case for a seperate work. 
}
\end{remark}

This paper is organised as follows. Section \ref{S2} covers preliminary results used throughout. 
Sections \ref{S3} and \ref{S4} are dedicated to proving Theorem \ref{mt-1}.  Given the complexity of the non-uniform convexity, we first present a model case of $k=4$ in Section \ref{S3} to illustrate the main steps and ideas, before addressing the general case in Section \ref{S4}. 
In Section \ref{S5} we study the homogeneous Monge-Amp\`ere equation \eqref{u-f=0} and prove Theorem \ref{mt-2}. Finally, in Section \ref{S6}, we establish the regularity for the convex envelope \eqref{u-ce} and prove Theorem \ref{mt-3}. 
An appendix includes some technical computations.

\vskip5pt 
\section{Preliminaries}\label{S2}

In this section, we summarise foundational results used throughout the paper.
Caffarelli's interior $C^{1,\alpha}$ estimate relies on the strict convexity of $u$ and the doubling condition for $f$. 
Specifically, if $f$ satisfies \eqref{dc} and $u^*$ is a convex solution of
	\begin{align*}
		\det D^2u^* &= f \quad\mbox{ in }\Omega^*, \\
		u^* &= 0 \quad\mbox{ on }\p\Omega^*,
	\end{align*}
where $B_1\subset\Omega^*\subset B_n$ and the centre of mass at zero, then
	\begin{equation}\label{daes}
	\begin{split}
		|u^*(x)| &\leq Cd^{1/n}(x,\p\Omega^*)|\inf_{\Omega^*}u^*| \\
			& \approx Cd^{1/n}(x,\p\Omega^*)\mu_{*}^{1/n}(\Omega^*),
	\end{split}
	\end{equation}
where $C>0$ depends on the doubling constant $C_b$, and $\mu_*$ is the Monge-Amp\`ere measure of $u^*$. 
For a detailed proof of \eqref{daes}, we refer the reader to \cite{C-1991,F-2017,LW14}.

Let $u$ be the convex solution of \eqref{u-fp} and $x_0\in\Omega$. An affine function $\ell_0$ is a support of $u$ at $x_0$ if $u(x_0)=\ell_0(x_0)$ and $u\geq\ell_0$ in $\Omega$. 
Define the contact set of $u$ at $x_0$ as
	\begin{equation}\label{ctset}
		\mathcal{C}_0=\{x\in\overline\Omega : u(x)=\ell_0(x)\}.
	\end{equation}  
By the estimate \eqref{daes} and an analogous argument as in \cite{C-1990}, we can show that either $\mathcal{C}_0$ is a singleton, or it cannot have extreme points in the interior of $\Omega$. Moreover, $u$ must be differentiable at $x_0$.

If $u$ is strictly convex at $x_0$ (i.e., $\mathcal{C}_0$ is a singleton), we may, by translating coordinates and subtracting an affine function, assume $x_0=0$ and $\ell_0=0$ is the support of $u$ at $0$. 
Define the sub-level set $S_h$ of $u$ at $0$ by
	\begin{equation*}
		S_h = \left\{x\in\Omega : u(x)< h \right\}.
	\end{equation*}
Let $h_0=\sup\{h>0 : S_h\Subset\Omega\}$. 
Applying \eqref{daes} to a normalised profile of $S_{h}$ with $h\leq h_0$,
we can derive that there is a constant $\sigma\in(0,1/2)$ depending on $n$ and $C_b$ such that the following estimates hold:
\begin{itemize}
\item[$(i)$] Balance estimate (\cite{C-1990}): Let $L=\overline{zz^*}\subset\overline{S_h}$ be a line segment passing through the origin with $z, z^*\in\p S_h$. Then $|z|\approx |z^*|$, namely
	\begin{equation}\label{sigma}
		\sigma |z^*| \leq |z| \leq \sigma^{-1}|z^*|. 
	\end{equation}

\item[$(ii)$] Decay estimate (\cite{C-1991}): For any $z\in\p S_h$,
	\begin{equation*}
		u\Big(\frac12z \Big) \leq \Big(\frac12-\sigma \Big) u(z).
	\end{equation*}
\end{itemize}
The decay estimate yields the $C^{1,\alpha}$ regularity of $u$ at $0$. 
Specifically, by choosing $\alpha>0$ such that $2^{-(1+\alpha)}=\frac12-\sigma$, we can, by induction, obtain
	\begin{equation}\label{inde}
		u(x) \leq \frac{2u(z)}{|z|^{1+\alpha}}|x|^{1+\alpha} \quad \text{ for any }\,x\in\overline{oz}.
	\end{equation}

For the Dirichlet problem \eqref{u-fp}, the solution $u$ is strictly convex if $f\geq c >0$ when $n=2$, and additionally $\varphi\in C^{1,\alpha}$ for $\alpha>1-\frac2n$ when $n\geq3$. 
However, in our case $f$ does not have a positive lower bound. 
Hence, $\mathcal{C}_0$ may contain a line segment with extreme points on $\p\Omega$, or $h_0$ could be very small (i.e., no positive lower bound on $\Omega'\Subset\Omega$). 

By a new observation using the boundary condition, Caffarelli, Tang and Wang obtained the following interior $C^{1,\alpha}$ regularity of $u$:
\begin{theorem}[\cite{CTW-2022}]\label{lem-3.2}
Let $u$ be the convex solution of \eqref{u-fp} with $f$ satisfying \eqref{dc}.
Assume $\Omega$ is bounded and convex, and $\p\Omega$ and $\varphi$ are $C^{1,\beta}$ for some $\beta>0$.
Then, for any $x_0\in\Omega'\Subset\Omega$, 
	\begin{equation*}
		0\leq u(x) - u(x_0) - Du(x_0) \leq C|x-x_0|^{1+\alpha} \quad\mbox{ in }\Omega,
	\end{equation*}
where $\alpha$ depends on $n, \beta, C_b$, and $C$ depends additionally on $\dist(\Omega',\p\Omega)$ and $\|\varphi\|_{C^{1,\beta}(\p\Omega)}$.
\end{theorem}
The boundary estimate is more subtle. 
By further assuming that $\p\Omega, \varphi \in C^{2,\beta}$ and $\p\Omega$ is uniformly convex, the boundary $C^{1,\alpha}$ regularity of $u$ was also obtained in \cite{CTW-2022}. 
In this paper, we aim to relax the uniform convexity condition on $\p\Omega$.
Given the complexity of the general case, we first demonstrate the main ideas and techniques using a model case in the next section.

\vskip5pt
\section{A model case of $k=4$}\label{S3} 

In this section, we assume that $\Omega\subset \R^2$ is a bounded, convex domain within the strip $\{0<x_2<1\}$, with $0\in\p\Omega$. Near the origin, the boundary $\p\Omega$ is given locally by
	\begin{equation}\label{lb-0}
		x_2 = \rho(x_1) = x_1^{4}.
	\end{equation}
First, we derive a global gradient estimate and introduce a local coordinate transformation, which will also be used later in the boundary $C^{1,\alpha}$ estimates. 

\begin{lemma}\label{thm-gra-est}
Let $u$ be a convex solution to \eqref{u-fp} with $f$ satisfying \eqref{f0}.
Assume that $\p\Omega$ satisfies \eqref{lb-0} and $\varphi\in C^{3,1}$, with $\varphi''(0)=0$.
Then, $u\in C^{0,1}(\overline\Omega)$ and 
	\begin{equation}\label{C1bound}
		|Du| \le C
	\end{equation}
for some positive constant $C$ depending only on $f_0, \Omega$ and $\|\varphi\|_{C^{3,1}(\p\Omega)}$.
\end{lemma}

We note that $\varphi\in C^{3,1}$ is necessary. For example, if $u(x_1,x_2)= -x_2^{1-\vep/4}$, then $\varphi=u|_{\p\Omega}\in C^{3,1-\vep}(\p\Omega)$ for any small $\vep>0$, and $\varphi''(0)=0$, but $u$ is not $C^{0,1}(\overline\Omega)$.

\begin{proof}
By the convexity of $u$ and the comparison principle, it is easy to see that $|u|_0\leq C$ in $\overline\Omega$, where $C$ depends on $\varphi, f_0$ and $\Omega$.  
By the Dirichlet boundary condition, in order to derive the gradient estimate of $u$, it suffices to show that
	\begin{equation}\label{ones}
		u_{\nu}(z)\leq C
	\end{equation} 
for $z\in\p\Omega$ near the origin, where $\nu$ 
is the unit outer normal of $\p\Omega$ at $z$, and $C$ is independent of the curvature of $\p\Omega$ at $z$.
By the comparison principle \cite{GT}, \eqref{ones} can be obtained if there exists a suitable sub-barrier function. 
When $\Omega$ is uniformly convex, a sub-barrier can be easily constructed using the defining function of $\Omega$. 
However, in the non-uniformly convex case \eqref{lb-0}, we need to set up new coordinates as follows. 

At $z\in\p\Omega$, the tangent line of $\p\Omega$ is $x_2=l_{z}(x_1) = 4z_{1}^3 (x_1 - z_{1}) + z_{1}^4$.
Introduce a coordinate transform 
	\begin{equation}\label{y=tx}
		y = \mathcal T(x) = \mathcal A x + \mathcal B,
	\end{equation}
where $\mathcal A =
\begin{pmatrix}
1 & 0 \\
- 4z_{1}^3 & 1
\end{pmatrix}$ and $\mathcal B = \begin{pmatrix}
- z_{1} \\
3 z_{1}^4
\end{pmatrix}$,
so that 
	\begin{equation}\label{y=tx-1}
			y_1 =  x_1 - z_1, \qquad
			y_2 = x_2 - l_z(x_1). 
	\end{equation}
Let $\tilde\Omega := \mathcal T (\Omega) \subset \{y_2 >0\}$.
This transform $\mathcal T$ maps $z$ to $0_y$ (the origin in $(y_1,y_2)$-coordinates), and $\p \tilde\Omega$ locally satisfies
	\begin{equation}\label{y2-r0}
		y_2  =\tilde \rho(y_1) = 6 z_{1}^2 y_1^2 + 4 z_{1} y_1^3 + y_1^4,
	\end{equation}
which implies
	\begin{equation}\label{y2-r3}
	\begin{split}
		\frac14 ( z_{1}^2 y_1^2 + y_1^4) \le y_2 \le 12 ( z_{1}^2 y_1^2 + y_1^4).
	\end{split}
	\end{equation}

Under the transform $\mathcal T$, the boundary data becomes $\tilde\varphi(y) = \varphi(\mathcal T^{-1} y)$ for $y\in\p\tilde\Omega$.  
We extend $\tilde\varphi$ from $\p\tilde\Omega$ to $\tilde D_h := \{y\in \tilde\Omega ~|~ \tilde \rho(y_1) < y_2 < h\}$ by setting
	\begin{equation}\label{extphi}
		\tilde\varphi(y_1,y_2) = \tilde\varphi(y_1, \tilde\rho(y_1))\quad\text{ for } \ (y_1,y_2)\in \tilde D_h ,
	\end{equation}
where $h>0$ is small.  
Using condition ($\mathcal P1$) and the convexity, we can verify that  
	\begin{equation}\label{var2-2}
		\tilde \varphi''(-z_1)= \tilde \varphi'''(-z_1) =0, 
	\end{equation}
and then, applying Taylor's expansion and the $C^{4,\beta}$ regularity of $\tilde\varphi$, we find
	\begin{equation}\label{var-2-3}
	\begin{split}
		\tilde\varphi^{(4)}(0)  &= \tilde\varphi^{(4)}(-z_1) + O(|z_1|^{\beta}), \\
		\tilde\varphi'''(0)  &= \tilde\varphi'''(-z_1) + z_1 \tilde\varphi^{(4)} (-z_1) + O(|z_1|^{1+\beta}) \\
					&=z_1 \tilde\varphi^{(4)}(0) + O(|z_1|^{1+\beta}), \\
		 \tilde\varphi''(0) & = \tilde\varphi''(-z_1) + z_1 \tilde\varphi'''(-z_1) + \frac{z_1^2}{2} \tilde\varphi^{(4)}(-z_1) + O(|z_1|^{2+\beta}) \\
 					& = \frac{z_1^2}{2} \tilde\varphi^{(4)}(0) + O(|z_1|^{2+\beta}).
	\end{split}
	\end{equation}
Therefore, by \eqref{y2-r0} we have  
	\begin{equation}\label{taylor-1}
	\begin{split}
		\tilde\varphi(y_1,\tilde\rho(y_1)) =&\ \tilde\varphi(0) + \tilde\varphi'(0)y_1 + \frac12 \tilde\varphi''(0)y_1^2 + \frac16 \tilde\varphi'''(0)y_1^3 \\
		&\ + \frac1{24} \tilde\varphi^{(4)}(0)y_1^4 + O(|y_1|^{4+\beta}) \\
		=&\  \tilde\varphi(0) + \tilde\varphi'(0)y_1 + \Big(\frac14 z_1^2 y_1^2 + \frac16 z_1 y_1^3 +\frac1{24} y_1^4 \Big)  \tilde\varphi^{(4)}(0) \\
		&\   + \Big(O(|z_1|^{2+\beta}) y_1^2 + O(|z_1|^{1+\beta})y_1^3 + O(|z_1|^{\beta})y_1^4 \Big) + O(|y_1|^{4+\beta}) \\
		=&\ \tilde\varphi(0) + \tilde\varphi'(0)y_1 + \frac1{24} \tilde\rho(y_1)   \tilde\varphi^{(4)}(0) + O \big(|z_1|^{\beta} + |y_1|^{\beta}\big) \tilde\rho(y_1),
	\end{split}
	\end{equation}
which thus implies
	\begin{equation}\label{g-v-up} 
		\tilde\varphi(y_1) - \tilde\varphi(0) - \tilde\varphi'(0)y_1 \le C_1 \tilde\rho(y_1), 
	\end{equation}
where the constant $C_1$ depends on $\|\varphi\|_{C^{3,1}(\p\Omega)}$.

\vskip5pt
With this setup, we can now construct the barrier function
	\begin{equation}\label{w-un-0}
		w(y) :=  \frac12 y_1^2 y_2^{q- \frac12} + Q y_2^{q} - M y_2 \quad\text{ in } \ \tilde D_h,
	\end{equation}
where $q>1$ and $Q, M$ are positive constants to be determined.
By differentiation, we have 
	\begin{equation*}
	\begin{split}
 		& w_{11} = y_2^{q-\frac12}, \\
 		& w_{12} = (q-\frac12) y_1y_2^{q-\frac32}, \\
 		& w_{22} = \frac12(q-\frac12)(q-\frac32) y_1^2y_2^{q-\frac52} + q(q-1) Q y_2^{q-2}.
	\end{split}
	\end{equation*}
Hence, in $\tilde D_h$ we obtain
	\begin{equation}\label{dw-0}
	\begin{split}
		\det D^2w & = y_2^{2q-3} \Big( q(q-1)Q y_2^{\frac12} - \frac12(q-\frac12)(q+\frac12)y_1^2 \Big) \\
			& \ge y_2^{2q-3} \Big( q(q-1)Q y_2^{\frac12} - (q-\frac12)(q+\frac12)y_2^{\frac12} \Big) \\
			& = y_2^{2q-\frac52} \Big( q(q-1)Q - (q-\frac12)(q+\frac12)  \Big) \ge f_0 \ge f = \det D^2 u,
	\end{split}
	\end{equation}
if choosing $q\in(1,\frac54)$, $Q=\frac{q}{q-1}$ and $h \le (4f_0)^{2/(4q-5)}$.

Next, we show that $w \le 0$ on $\p \tilde D_h$.
Indeed, on $\p \tilde D_h\cap\{y_2<h\}$, by \eqref{y2-r3}, we get
	\begin{equation}\label{w-0}
	\begin{split}
		w & = \frac12 y_1^2 y_2^{q- \frac12} + Q y_2^{q} - M y_2 \\
			& \le  y_2^{q} + Q y_2^{q} - M y_2 \\
			& = y_2 \Big( (1+Q)y_2^{q-1} - M \Big) \\
			& \le y_2 \Big( (1+Q) h^{q-1} - M \Big) \le 0
	\end{split}
	\end{equation}
if we choose $M= (1+Q) h^{q-1}$.
On $\p \tilde D_h\cap\{y_2=h\}$, by convexity we also have $w\leq 0$. 

Last, fix the constant $h=(4f_0)^{2/(4q-5)}$ and let
	\begin{equation*}
		\hat w(y) =  w(y) + \tilde\varphi(0_y) + \tilde\varphi'(0_y)y_1 - \Big(C_1 + \frac{\sup_\Omega|u|}{h} \Big) y_2.
	\end{equation*}
It is easy to check that $\det D^2 \hat w \ge \det D^2 u$ in $ \tilde D_h$ from \eqref{dw-0}, and $\hat w \le u$ on $\p \tilde D_h$ from \eqref{g-v-up} and \eqref{w-0}.
Hence, 
	\begin{equation*}
		u_{\nu}(z)\leq \left|\frac{\p}{\p y_2}u(0_y)\right| \le |\hat w_{y_2}(0_y)|= M + C_1 + \frac{\sup|u|}{h} =:C
	\end{equation*} 
as desired in \eqref{ones}, where the constant $C$ is independent of the curvature of $\p\Omega$ at $z$.
\end{proof}

\begin{corollary}\label{lemma5.1}
Assume that $\p\Omega \in C^1$ and $u|_{\p\Omega} \in C^1$.
Then, $Du$ is well defined on $\p\Omega$ in the sense that for every $z\in\p\Omega$, there is a unique affine function $\ell(x)=a_0+\sum_{i=1}^2a_i(x_i-z_i)$ such that
	\begin{equation}\label{bdc1} 
		|u(x) - \ell(x)| = o(|x-z|)\quad\text{ as } x\to z. 
	\end{equation}
\end{corollary}

\begin{proof}
Indeed, this result holds in any dimension provided $u\in C^{1,\alpha}(\Omega)\cap C^{0,1}(\overline\Omega)$. See \cite[Lemma 4.1]{CTW-2022}. For completeness, we provide a brief proof. Choosing the proper coordinates we assume $z=0$ and $\Omega\subset\{x_2>\rho(x_1)\}$ with $\rho(0)=0$, $\rho'(0)=0$. Since $u|_{\p\Omega} \in C^1$, by differentiation we have
	\[ \left.\frac{d}{dx_1}\right|_{x_1=0}u(x_1,\rho(x_1)) = u_1(0) + u_2(0)\rho'(0) = u_1(0),\]
which is well-defined, where the last equality is due to $u\in C^{0,1}(\overline\Omega)$. 
By the convexity of $u$, $u_2(te_2)$ is monotone in $t$. Hence we can define $u_2(0)=\lim_{t\to0^+}u_2(te_2)$.
Therefore, the unique affine function $\ell(x)=u(0)+u_1(0)x_1+u_2(0)x_2$ satisfies \eqref{bdc1}.
\end{proof}

\vskip5pt
\subsection{Boundary $C^{1,\alpha}$ estimates} 
We are now ready to derive the boundary $C^{1,\alpha}$ estimate.

\begin{theorem}\label{thm5.1}
Let $u$ be a convex solution to \eqref{u-fp} with $f$ satisfying \eqref{dc} and \eqref{f0}.
Assume that $\p\Omega$ satisfies \eqref{lb-0}, and that $\varphi\in C^{4,\beta}$ for some $\beta\in(0,1]$, with $\varphi''(0)=0$.
Then, for any $z\in\p\Omega$,  
	\begin{equation}\label{bdryC1a}
		|Du(x)-Du(z)| \leq C|x-z|^\alpha \qquad\text{ for any } x\in\overline\Omega  \text{ near } z,
	\end{equation}
for some {$\alpha \in (0,\beta/4)$} and constant $C>0$ that depend only on $C_b, f_0, \beta,\Omega$ and $\|\varphi\|_{C^{4,\beta}(\p\Omega)}$.
\end{theorem}

The proof of Theorem \ref{thm5.1} is divided into three lemmas.
In Lemma \ref{lemma5.2} we first establish the H\"older continuity for tangential derivatives, and then in Lemmas \ref{lemma5.3} and \ref{lemma5.4}, we address the normal derivative. 

\begin{lemma}\label{lemma5.2}
Let $\tau$ be a unit tangential vector of $\p\Omega$ at $z$,
and let $\varphi\in C^{1,\alpha}$ for some $\alpha\in(0,1]$.
Then,
	\begin{equation}\label{tanest}
		|u_{\tau}(x) - u_{\tau}(z)| \le C |x-z|^{\frac{\alpha}{1+\alpha}}
	\end{equation}
for any point $x\in\overline\Omega$ near $z$, where the constant $C>0$ depends only on $\alpha,\Omega, \sup|u|$ and $\|\varphi\|_{C^{1,\alpha}(\p\Omega)}$.
\end{lemma}

\begin{proof}
The proof follows from \cite{CTW-2022}. For completeness, we provide a brief outline below.  

By changing coordinates, we may assume that $z=0$, $\Omega\subset\{x_2>0\}$ and $\tau=e_1$ at $0$. By Corollary \ref{lemma5.1}, we may assume that $\ell\equiv0$ is the support of $u$ at $0$ such that $Du(0)=0$. 
Thus, it suffices to show that 
	\begin{equation}\label{bdalpha1}
		|u_{1}(x)|\leq C|x|^{\frac{\alpha}{1+\alpha}} \ \ \mbox{ for $x\in\overline\Omega$ near the origin}. 
	\end{equation}
Let $L$ be the straight line passing through $x$ and parallel to $e_1$ and $L\cap\partial\Omega=\{x_l, x_r\}$ such that $(x_r-x_l)\cdot e_1=|x_r-x_l|$.

Since $x_r\in\partial\Omega$, we have $|\partial_\xi u(x_r)|=|\partial_\xi\varphi(x_r)| \leq C|x_r|^\alpha$, where $\xi$ is the unit tangential vector at $x_r$. 
Let $\theta$ be the angle between $\xi$ and $e_1$. Since $\partial\Omega$ is smooth, we have $\theta \leq C|x_r|$. Hence,
	\begin{equation}\label{bdal2}
		|u_1(x_r)| \leq |u_\xi(x_r)|\cos\theta + |u_\gamma(x_r)|\sin\theta \leq C|x_r|^\alpha,
	\end{equation}
where $\gamma$ is the unit inner normal of $\p\Omega$ at $x_r$.

To prove \eqref{bdalpha1}, we may assume that $u_1(x)\geq0$; otherwise, we can replace $e_1$ by $-e_1$. 
By the convexity of $u$, we have $u_1(x) \leq u_1(x_r)$. 
If $|x|\geq\frac12|x_r|$, then by \eqref{bdal2}, we easily obtain
	\begin{equation}\label{bdal3}
		u_1(x) \leq C|x|^\alpha \leq C|x|^{\alpha/(1+\alpha)}.
	\end{equation}

If $|x|<\frac12|x_r|$, we may assume that $x_1\geq0$; otherwise, by the convexity of $u$, we have $0\leq u_1(x) \leq u_1(0,x_2)$ and it suffices to prove \eqref{bdalpha1} at $\bar x=(0,x_2)$. 
For a small $t>0$ such that $x+te_1\in\Omega$, by the convexity of $u$, we find that
	\begin{equation*}
		u_1(x) \leq \frac{u(x+te_1) - u(x)}{t} \leq \frac{u(x+te_1)}{t}.
	\end{equation*}
By Lemma \ref{thm-gra-est} and the assumption $\varphi\in C^{1,\alpha}$, we obtain
	\begin{equation*}
		u(x+te_1) \leq C_1(x_1+t)^{1+\alpha} + C_2x_2,
	\end{equation*}
where $C_1$ depends on $\partial\Omega$ and $\|\varphi\|_{C^{1,\alpha}(\partial\Omega)}$, and $C_2$ is the gradient bound in \eqref{C1bound}. Thus,
	\begin{equation}\label{bdal4}
		u_1(x) \leq C_1\frac{(x_1+t)^{1+\alpha}}{t} + C_2\frac{x_2}{t} =: I_1+I_2.
	\end{equation}
Choose $t = \max\{x_1, x_2^{{1}/{(1+\alpha)}}\}$.
Since $t \geq x_2^{{1}/{(1+\alpha)}}$, the second term $I_2 \leq C_2x_2^{{\alpha}/{(1+\alpha)}}\leq C_2|x|^{{\alpha}/{(1+\alpha)}}$. Since $t \geq x_1$, the first term $I_1 \leq 4C_1t^\alpha$.
If $x_1\geq x_2^{{1}/{(1+\alpha)}}$, then $t=x_1$, and thus $I_1\leq 4C_1x_1^\alpha\leq 4C_1|x|^{\alpha/(1+\alpha)}$.
If $x_1<x_2^{{1}/{(1+\alpha)}}$, then $t=x_2^{{1}/{(1+\alpha)}}$, and thus $I_1 \leq 4C_1x_2^{\alpha/(1+\alpha)} \leq 4C_2|x|^{\alpha/(1+\alpha)}$. Therefore, by \eqref{bdal4} we obtain
	\begin{equation}\label{bdal5}
		u_1(x) \leq C|x|^{\alpha/(1+\alpha)}.
	\end{equation}
Note that we may replace $t$ with $ct$ for a small constant $c>0$ to ensure $x+te_1\in\Omega$. 
	
Combining \eqref{bdal3} and \eqref{bdal5}, we thus achieve \eqref{bdalpha1} as desired. 	
\end{proof}

\begin{remark}\label{rmk3.1}
The estimate \eqref{tanest} holds for general convex domains in any dimension, provided that $\p\Omega\in C^{1,\alpha}$. See \cite{CTW-2022}.
\end{remark}

Given the assumption \eqref{lb-0}, $\p\Omega$ has an isolated degenerate point at the origin.
For any small constant $\delta_1>0$, there exists $\delta_2>0$ such that $\kappa(x) \ge \delta_2$ for all $x_2 \ge \delta_1$, where $\kappa(x)$ is the curvature of $\p\Omega$ at $x\in\p\Omega$.
Therefore, it suffices to prove Theorem \ref{thm5.1} for all $z\in \p\Omega \cap \{x_2 < \delta_1\}$ for some small positive constant $\delta_1$ to be determined later.

\begin{lemma}\label{lemma5.3}
Let $u$ be the convex solution to \eqref{u-fp}.
Under the hypotheses of Theorem \ref{thm5.1}, there exists an $\alpha\in(0, \beta/4)$ such that, for any $z \in \p\Omega$ near the origin,
	\begin{equation}\label{bdnm1}
		|u_{\nu}(x) - u_{\nu}(z)| \le C |x-z|^{\alpha} \quad \text{ for any } x\in\partial\Omega \  \text{ near $z$},
	\end{equation}
where $\nu$ is the unit inner normal of $\partial\Omega$ at $z$,  and the constants $\alpha$ and $C$ depend only on $f_0, \beta$ and $\|\varphi\|_{C^{4,\beta}(\p\Omega)}$, but independent of $z$.
\end{lemma}

\begin{proof}
At $z\in\partial\Omega$, we make the coordinate change $y=\mathcal{T}(x)$ as defined in \eqref{y=tx} and \eqref{y=tx-1}, such that $\mathcal{T}(z)=0_y$ represents the origin in the $y$-coordinates. Locally, $\tilde\Omega=\mathcal{T}(\Omega)\subset\{y_2>0\}$ satisfies \eqref{y2-r0} and \eqref{y2-r3}, while the boundary function $\tilde\varphi$ satisfies \eqref{var2-2}--\eqref{g-v-up}. 
Unless stated otherwise, the subsequent calculations will be conducted in the $y$-coordinates. 
It suffices to show that
	\begin{equation}\label{bdnm2}
		|u_2(\tilde x) - u_2(0)| \leq C |\tilde x|^\alpha
	\end{equation}
for any $\tilde x\in\partial\tilde\Omega$ near the origin. 
In the following, we fix such an $\tilde x$.
By subtracting a linear function we assume that $\{y_{3} = 0\}$ is the support plane of $u$ at the point $\tilde x$, such that $u(\tilde x)=0$, $Du(\tilde x)=0$ and $u\geq0$.
We shall construct a barrier function $w$ at the origin to prove \eqref{bdnm2}, specifically 
	\begin{equation*}\label{bdnm3}
		|u_2(0)| \leq C |\tilde x|^\alpha.
	\end{equation*}

Let $\tilde D_h = \{y \in \tilde\Omega ~|~ 0< y_2 < h\}$ with $h>0$ being small, 
and define the function  
	\begin{equation}\label{w-un}
		w(y) =  \frac12 y_1^2 y_2^{q- \frac12} + \frac{q}{q-1} y_2^{q} - B  y_2,
	\end{equation}
where $q\in(1,\frac54)$ and $B>0$ are constants to be determined. 
By computation in \eqref{dw-0}, we have
	\begin{equation*}
		\det\,D^2w = y_2^{2q-3}\left( q^2y_2^{\frac12} - \frac12(q^2-\frac14)y_1^2 \right).
	\end{equation*}
From \eqref{y2-r3}, $y_1^2\leq 2y_2^{1/2}$. Thus, we obtain
	\begin{equation}\label{compMA}
		\det\,D^2w \geq \frac14y_2^{2q-\frac52} \geq f_0 \geq \det\,D^2u \quad\text{ in }\tilde D_h,
	\end{equation}
provided $h \le (4f_0)^{2/(4q-5)}$. 

To compare $w$ with $u$ on $\partial \tilde D_h$, we define
	\begin{equation}\label{nhatu} 
		\hat u(y) = u(y) - \left( u(0) + u_{1}(0) y_1 \right).
	\end{equation}
Since $\varphi\in C^{4,\beta}$ and $Du(\tilde x)=0$, by \eqref{bdal2} we have
	\begin{equation*}
		0 \le u(0) \le C|\tilde x|^2, \quad |u_{1}(0)| \le C|\tilde x|.
	\end{equation*}
On $\p \tilde D_h \cap \{y_2 =h\}$, since $u\geq0$ and $|y_1|\leq \sqrt{2}h^{1/4}$, we obtain
	\begin{equation*}
		\hat u(y) \geq -C|\tilde x|^2 - C|\tilde x|h^{\frac14} \geq - 2C|\tilde x|h^{\frac14},
	\end{equation*}
if we choose $h$ such that $|\tilde x|\leq h^{1/4}$.
Meanwhile, since $q>1$, for small $h>0$ and sufficiently large $B>0$, we get
	\begin{equation}\label{wB1}
		w(y) \leq \frac{2q-1}{q-1}h^{q} - Bh \leq -\frac12Bh.
	\end{equation}
Hence, by choosing $B\geq C_1|\tilde x|h^{-3/4}$, we obtain
	\begin{equation}\label{combdr1}
		w \leq \hat u \quad\text{ on } \p \tilde D_h \cap \{y_2 =h\}. 
	\end{equation}

On $\p \tilde D_h \cap \{y_2 <h\}$, from \eqref{nhatu}, we have $\hat u=\hat \varphi$ with
	\begin{equation}\label{hatphi}
	    \hat\varphi(y) = \tilde\varphi(y)-\left(\tilde\varphi(0)+\tilde\varphi'(0)y_1\right),
    \end{equation}
where $\tilde\varphi$ has been extended in $\tilde D_h$ as per \eqref{extphi}. 
To estimate $\hat u$ on the boundary $\p\tilde\Omega$, we consider the following two cases:

\noindent\textbf{Case I:} Suppose $|z_1| > h^\gamma$ for some $\gamma\in(0,1/6)$.
Since $\delta_1$ and $h$ are small, it follows from \eqref{y2-r3} that
    \begin{equation}\label{y2-r4}
	     y_2^{\frac12} \le |y_1| \le  2 |z_1|^{-1} y_2^{\frac12} \le 2 y_2^{\frac12-\gamma} 
	     \quad \text{for }   (y_1,y_2)\in \p\tilde \Omega \cap \{y_2 < h\}.
    \end{equation}
Since $\{y_3=0\}$ is the support plane of $u$ at $\tilde x$, the convexity of $u$ implies $\tilde\varphi''(\tilde x)\geq 0$, and hence $\tilde\varphi''(0)\geq -C|\tilde x|$.
Using Taylor's expansion, we obtain
    \begin{equation*}
		\hat\varphi(y) = \frac12 \tilde\varphi''(0)y_1^2  + O(|y_1|^{3}) \ge -C(|\tilde x| y_1^2 +|y_1|^3)
    \end{equation*}
for some constant $C>0$ depending on $\|\varphi\|_{C^{2,1}(\p\Omega)}$.

Now, by choosing $h$ such that $|\tilde x|\leq h^{\frac12}$ and noting that $\gamma<\frac16$, it follows from \eqref{y2-r4} that on $\p \tilde D_h \cap \{y_2<h\}$, 
    \begin{equation}\label{hat-u-1}
    	\begin{split}
    		\hat u(y) =\hat\varphi(y) &\ge -C (h^{\frac12}y_1^2 + y_2^{\frac32-3\gamma}) \\
    		&\ge -C (h^{\frac12}|z_1|^{-2}y_2 + h^{\frac12-3\gamma}y_2) \\
    		&\ge -C h^{\frac12-3\gamma}y_2
    	\end{split}
	\end{equation}
for a different $C>0$, where the last inequality uses the assumption $|z_1| > h^\gamma$.

\noindent\textbf{Case II:}
Suppose $|z_1| \le h^\gamma$.
By \eqref{taylor-1} and \eqref{hatphi}, we have
\begin{equation}\label{hathat}
	\hat\varphi(y) = \frac{1}{24} \tilde\varphi^{(4)}(0) y_2 + O(|z_1|^\beta + |y_1|^\beta) y_2.
\end{equation}
Similar to \eqref{var-2-3}, by Taylor's expansion we have
    \begin{equation*} 
    		\tilde\varphi''(\tilde x) = \frac{|\tilde x_1+z_1|^2}{2}\tilde\varphi^{(4)}(-z_1) + O(|\tilde x_1+z_1|^{2+\beta}). 
    \end{equation*}
Recall that $\{y_3=0\}$ is the support plane of $u$ at $\tilde x$ and $\tilde\varphi''(\tilde x)\geq0$. 
From above, we have $\tilde\varphi^{(4)}(-z_1)\geq -C(|\tilde x_1|^\beta+|z_1|^\beta)$, and hence $\tilde\varphi^{(4)}(0) \ge - C(|\tilde x_1|^\beta+|z_1|^\beta)$ due to $\tilde\varphi\in C^{4,\beta}$.
Consequently, by \eqref{hathat}, we have
    \begin{equation*} 
	\hat u(y) =\hat\varphi(y) \ge - C(|\tilde x_1|^\beta + |z_1|^{\beta} + |y_1|^\beta) y_2 \quad \text{on }\ \partial\tilde D_h\cap \{y_2<h\}.
    \end{equation*}
From \eqref{y2-r3}, we get $|y_1|\leq Cy_2^{1/4}\leq Ch^{1/4} \leq Ch^\gamma$ since $\gamma<1/6$. Then by the assumption $|z_1|\leq h^\gamma$, as far as $|\tilde x_1| \leq h^{1/6}$, we can obtain 
    \begin{equation}\label{hat-u-2}
	\hat u(y) \ge - Ch^{\gamma\beta} y_2 \quad \text{on }\ \partial\tilde D_h\cap \{y_2<h\}.
    \end{equation}
 
Therefore, combining \eqref{hat-u-1} and \eqref{hat-u-2} we obtain
\begin{equation*}
	\hat u(y)  \ge - C h^{\hat\alpha} y_2 \quad \text{on} \ \p \tilde D_h \cap \{y_2<h\},
\end{equation*}
where $\hat\alpha = \min\{\frac12-3\gamma, \gamma\beta\}$.

Recalling the definition of $w$ in \eqref{w-un} and noting that $y_1^2\leq2y_2^{1/2}$ for $y\in\partial\tilde\Omega$, we have
	\begin{equation}\label{wB2}
		w(y) \leq \left(\frac{2q}{q-1}h^{q-1} - B\right)y_2 \quad \text{on} \ \ \p \tilde D_h \cap \{y_2<h\}.
	\end{equation}
Therefore, by choosing $B\geq \frac{2q}{q-1}h^{q-1}+ Ch^{\hat\alpha}$, we obtain
	\begin{equation}\label{combdr2}
		w \leq \hat u \quad \text{on} \ \ \p \tilde D_h \cap  \{y_2<h\}.
	\end{equation}
   
Finally, set $q=1+ \hat\alpha$ in \eqref{w-un} and define $B=C_1|\tilde x|h^{-3/4} + Ch^{\hat\alpha}$ for a different $C>0$. From \eqref{combdr1} and \eqref{combdr2}, we have $w(0)=\hat u(0)$ and 
	\begin{equation*}
		w \leq \hat u \quad\text{on } \ \p\tilde D_h.
	\end{equation*} 
Applying \eqref{compMA} and the comparison principle, we conclude that
	\begin{equation*}
		u_2(0)=\hat u_2(0) \geq w_2(0) \geq - C_1|\tilde x|h^{-3/4} - Ch^{\hat\alpha}.
	\end{equation*}
Choosing $h$ such that $|\tilde x| = h^{\frac34+\hat\alpha}$, we then obtain
	\begin{equation*}
		u_2(0) - u_2(\tilde x) \geq - Ch^{\hat\alpha}.
	\end{equation*} 
Similarly, we have $u_2(\tilde x) - u_2(0) \geq - Ch^{\hat\alpha}$. 
Hence, we obtain \eqref{bdnm2} for $\alpha = \frac{4\hat\alpha}{3+4\hat\alpha}$.
This completes the proof. 
\end{proof}

\begin{lemma}\label{lemma5.4}
Let $u$ be the solution to \eqref{u-fp}.
Assume the same conditions as in Lemma \ref{lemma5.3}.
Then, for any $z\in\p\Omega$ near the origin and $h>0$ small,
	\begin{equation*}\label{u-x2-1}
		|u_{\nu}(z+h\nu) - u_{\nu}(z)| \le C h^{\alpha}  
	\end{equation*}
for some $\alpha \in(0,\beta/4)$, where $\nu$ is the unit inner normal of $\p\Omega$ at $z$, and the constants  $\alpha, C$ depend only on $f_0, \beta$ and $\|\varphi\|_{C^{4,\beta}(\p\Omega)}$, but independent of $z$.
\end{lemma}

\begin{proof}
At $z\in\p\Omega$, we apply the coordinate change $y=\mathcal{T}(x)$ in \eqref{y=tx}, which positions $\mathcal{T}(z)=0$ at the origin in the $y$-coordinate. Locally, $\tilde\Omega=\mathcal{T}(\Omega)\subset\{y_2>0\}$ satisfies \eqref{y2-r0}--\eqref{y2-r3}, and the boundary function $\tilde\varphi$ satisfies \eqref{var2-2}--\eqref{g-v-up}. 

By subtracting a linear function, we can assume that $\{y_{3} = 0\}$ is the support plane of $u$ at $0$ such that $u(0)=0$ and $Du(0)=0$.
It suffices to prove
	\begin{equation}\label{u-y2-1}
		u (h e_2) \le C h^{1+\alpha} 
	\end{equation}
for small $h>0$.
The proof is divided into two cases based on the geometry of $\partial\Omega$ at $z$: Case I, when $z$ is far from the degenerate point (locally $\partial\Omega$ is more like a quadratic curve); and Case II, when $z$ is close to the degenerate point (locally $\partial\Omega$ is more like a quartic curve).  

\noindent\textbf{Case I:}
Assume $|z_1| > h^\gamma$ for some small $\gamma>0$ to be determined. 
Let $\tilde D_h=\tilde\Omega\cap\{y_2<h\}$ and $\bar\varphi:=u|_{\p\tilde\Omega}$.
By Taylor's expansion, we have
	\begin{equation}\label{5.40-1} 
		\bar\varphi(y) = \frac12 \bar\varphi''(0)y_1^2 + O(|y_1|^{3}) \quad\text{ on $\p \tilde D_h \cap \p\tilde\Omega$}.
	\end{equation}
Note that $\bar a_2 := \bar\varphi''(0)  \ge 0$ due to the convexity of $u$.

 \emph{Case $I_i$:}
Assume $\bar a_2 \le 2 h^{\alpha+2\gamma}$ for some $\alpha>0$.
For small $h>0$, let $t^+>0$ and $t^-<0$ be such that $\tilde \rho(t^+ ) = \tilde \rho(t^- )=h$.
Then, $(t^\pm, h) \in \p \tilde D_h \cap \p\tilde\Omega$.
By \eqref{y2-r3}, similar to \eqref{y2-r4} we have
	\begin{equation}\label{t-h-2}
		h^{1/2} \leq |t^\pm| \le 2 h^{1/2-\gamma},
	\end{equation}
provided $\delta_1$ and $h$ are sufficiently small. 
By \eqref{5.40-1} and the convexity of $u$, we obtain
	\begin{equation*}\label{c111}
		\begin{split}
		u(h e_2)  & \le  \max\{\bar\varphi(t^+), \bar\varphi(t^-)\}   \le \frac12 \bar a_2 |t^\pm|^2 + C |t^\pm|^3  \\
			&  \le 4 h^{1+\alpha} + C h^{3/2-3\gamma} \le C h^{1+\alpha}
		\end{split}
	\end{equation*}
provided $\alpha < \frac12- 3\gamma$.

 \emph{Case $I_{ii}$:}
Assume $\bar a_2>2h^{\alpha+2\gamma}$. To distinguish $\alpha$ here from the one in Lemma \ref{lemma5.3}, we denote the $\alpha$ in \eqref{bdnm1} by $\alpha_0$. 
Suppose, contrary to \eqref{u-y2-1}, that there exists a small $h > 0$ such that $u(h e_2) \ge h^{1+\alpha}$.
We will derive a contradiction using a modified barrier function $w$ from \eqref{w-un}.

On $\p\tilde D_h\cap\{y_2<h\}$, we have $y_2\leq |y_1|^2 \leq 4|z_1|^{-2}y_2 < 4h^{-2\gamma}y_2$ by \eqref{y2-r4}.
Then, using \eqref{5.40-1} and the assumption $\bar a_2>2h^{\alpha+2\gamma}$, we get
	\begin{equation}\label{bdest11}
	\begin{split}
		u(y) = \bar\varphi(y) &\ge \frac{1}{2}\bar a_2 |y_1|^2 - C|y_1|^{3} \\
			&\ge h^{\alpha+2\gamma} |y_1|^2 - C h^{-3\gamma} y_2^{\frac32} \\  
			&\ge h^{\alpha+2\gamma} y_2 - Ch^{\frac12-3\gamma}y_2  \ge \frac{1}2 h^{\alpha+2\gamma} y_2,
	\end{split}
	\end{equation}
provided $\alpha < \frac12 - 5\gamma$.

To estimate $u$ on $\p\tilde D_h\cap\{y_2=h\}$, by Lemma \ref{lemma5.3}, we have
	\begin{equation}\label{le5.3-}
		u_{2}(y_1,\tilde\rho(y_1)) \geq -C |y_1|^{\alpha_0}.
	\end{equation} 
If $y\in\p\tilde D_h\cap\{y_2=h\}$ and $|y_1| > h^{\frac12 + 2\alpha+2\gamma}$, then by \eqref{5.40-1} and \eqref{le5.3-}, we get 
	\begin{equation}\label{5.1--}
	\begin{split}
		u(y) &\ge u(y_1, \tilde\rho(y_1))  -C |y_1|^{\alpha_0} h \\
		     &\ge \frac{1}{2}\bar a_2 |y_1|^2 - C|y_1|^{3} -C |y_1|^{\alpha_0} h \\
			 &\ge h^{1+5\alpha+6\gamma} - Ch^{3(\frac12-\gamma)} - C  h^{1+ \alpha_0(\frac12-\gamma)} \ge \frac12 h^{1+5\alpha+6\gamma},
	\end{split}
	\end{equation}
provided $5\alpha+9\gamma < \frac12$ and $5\alpha+6\gamma < \alpha_0(\frac12-\gamma)$. 
 
If $y \in \p \tilde D_h \cap \{y_2 = h\}$ and $|y_1| \le h^{\frac12 + 2\alpha+2\gamma}$, then by the convexity of $u$, \eqref{5.40-1}, and \eqref{t-h-2}, we obtain
	$$
		|\p_{y_1} u(y)| \le \frac{\sup_{\tilde D_h} u}{h^{1/2} - h^{\frac12 + 2\alpha+2\gamma} }   \le  \frac2{h^{1/2}} \Big( \frac12 \bar a_2 h^{1-2\gamma} + C h^{\frac32-3\gamma} \Big)
\le C h^{\frac12-2\gamma}.
	$$
Thus, by the assumption $u(he_2)\geq h^{1+\alpha}$, we have
	\begin{equation}\label{5.2--}
		u(y_1,h) \ge u(he_2) - |\p_{y_1} u | |y_1| \ge h^{1+\alpha}- C  h^{1+ 2\alpha} \ge \frac12 h^{1+\alpha}.
	\end{equation}
Therefore, from \eqref{5.1--} and \eqref{5.2--}, it follows that
	\begin{equation}\label{5.3--}
		u(y) \ge \frac12 h^{1+5\alpha+6\gamma} = \frac12 h^{5\alpha+6\gamma} y_2 \quad\text{ on }\ \p \tilde D_h \cap \{y_2 =h\}.
	\end{equation}

Combining \eqref{bdest11} and \eqref{5.3--}, we conclude that
	\begin{equation*}
		u(y) \geq \frac12 h^{5\alpha+6\gamma} y_2   \quad\text{ on }\  \p \tilde D_h,
	\end{equation*}
provided $5\alpha+9\gamma < \frac12$ and $5\alpha+6\gamma < \alpha_0(\frac12-\gamma)$.
This is achievable, for instance, by choosing $\gamma\leq \frac{1}{20}\alpha_0$ and $\alpha\leq \frac{1}{50}\alpha_0$. 

Let $w$ be the function in \eqref{w-un} with $B=\frac{2q}{q-1}h^{q-1}$ and $q=1+ \frac{\alpha_0}2$. From \eqref{wB1} and \eqref{wB2}, we see that $w\leq 0$ on $\partial\tilde D_h$, provided $h>0$ is small.  
By \eqref{compMA} and applying the comparison principle to $w$ and $u- \frac12 h^{5\alpha+6\gamma} y_2$ in $\tilde D_h$, since $5\alpha+6\gamma < \frac{\alpha_0}2 = q-1$, we obtain
	\begin{equation*}
		\partial_{y_2}u(0)  \ge  \frac12 h^{5\alpha+6\gamma} - \frac{2q}{q-1}h^{q-1} >0,
	\end{equation*}
if $h>0$ is small.  
This contradicts the assumption that $Du(0)=0$.

\vskip5pt
\noindent\textbf{Case II:}
Suppose $|z_1| \le h^\gamma$ for the small $\gamma>0$ determined in {Case I}. 
Recall that $u=\bar\varphi$ on $\p\tilde\Omega$ with $\bar\varphi$ as in \eqref{5.40-1}, and $\bar a_2 = \bar\varphi''(0) \ge 0$.
Let $\bar a_4 := \bar\varphi^{(4)}(0)$.
It can be verified by calculation that $\bar\varphi''(-z_1)=\bar\varphi'''(-z_1)=0$. 
Then, similarly to \eqref{var-2-3} and \eqref{taylor-1}, we obtain
	\begin{equation}\label{barphi}
		\bar\varphi(t, \tilde\rho(t) ) = \frac{1}{24} \bar a_4 \tilde\rho(t) + O \big(|z_1|^{\beta} + |t|^{\beta}\big) \tilde\rho(t) \quad\text{ for small $|t|$. }
	\end{equation}

 \emph{Case $I\!I_{i}$:} Suppose $|\bar a_4| \le  24  h^\alpha$.
Then by \eqref{barphi}, we have
	\begin{equation}\label{6.38}
		\bar\varphi(t, \tilde\rho(t) ) \le C (h^\alpha + |z_1|^\beta + |t|^\beta) \tilde\rho(t).
	\end{equation} 
As in Case $I_i$, assume $(t^\pm, h)\in \p\tilde\Omega$.
By \eqref{y2-r3}, we obtain
	\begin{equation}\label{t-h-1}
		\frac14 h^{1/2-\gamma} \le |t^\pm| \le \sqrt{2} h^{1/4},
	\end{equation} 
where the second inequality follows directly from \eqref{y2-r3}, while for the first inequality, we used $|t^\pm| \leq \sqrt{2} h^{1/4} \leq \sqrt{2} h^\gamma$ since $\gamma<1/4$ is small.
	
By the convexity of $u$ and \eqref{6.38}, we get
	\begin{equation*}
	\begin{split}
		u(h e_2)  &\le  \max\{\bar\varphi(t^+), \bar\varphi(t^-)\} \\
			& \le C (h^{\alpha} + h^{\gamma\beta}+h^{\beta/4})h \le C h^{1+\alpha},
	\end{split}
	\end{equation*}
provided $\alpha < \gamma\beta$.

 \emph{Case $I\!I_{ii}$:}
Now suppose $|\bar a_4|>  24  h^\alpha$.
Since $\bar a_2\geq0$, we conclude $\bar a_4 > 24  h^\alpha$ by \eqref{var-2-3}.
In fact, by \eqref{var-2-3} and $\bar a_2\geq0$, we get $\bar a_4 > -C|z_1|^\beta$. Since $|z_1|\leq h^\gamma$, we then have $\bar a_4 > -Ch^{\gamma\beta}$. If $\bar a_4<-24h^\alpha$, it would lead to contradiction when $h>0$ is sufficiently small, provided $\alpha<\gamma\beta$. Hence, we obtain $\bar a_4 > 24  h^\alpha$.
Similarly to Case $I_{ii}$, we suppose to the contrary that there exists a small $h > 0$ such that $u(h e_2) \ge h^{1+\alpha}$.

On $\p \tilde D_h \cap \{y_2 <h\}$, by \eqref{barphi} and \eqref{t-h-1}, we have
	\begin{equation}\label{u22p}
	\begin{split}
		u(y) = \bar \varphi(y) &\ge \frac1{24} \bar a_4 y_2 - C( |z_1|^\beta + |y_1|^\beta) y_2 \\
			&\ge h^{\alpha} y_2 - C(h^{\gamma\beta} + h^{\beta/4}) y_2 \\
			&\ge \frac12 h^\alpha y_2,
	\end{split}
	\end{equation}
provided $\alpha < \gamma\beta$.

Next, we estimate $u$ on $\partial\tilde D_h\cap\{y_2=h\}$.
If $y \in \partial\tilde D_h\cap\{y_2=h\}$ with $|y_1|$ sufficiently large such that $\tilde\rho(y_1) > h^{1+ 8\alpha}$, we find, similarly to \eqref{5.1--}, using \eqref{barphi}, \eqref{t-h-1} and Lemma \ref{lemma5.3}, that  
	\begin{equation}\label{uII1}
	\begin{split}	
		u(y) & \geq \bar\varphi(y_1,\tilde\rho(y_1)) - C|y_1|^{\alpha_0} h \\
		     & \geq \frac{1}{24} \bar a_4 \tilde\rho(y_1) -C \big(|z_1|^{\beta} + |y_1|^{\beta}\big) \tilde\rho(y_1) - C  h^{1+ \alpha_0/4} \\
			&\geq h^{1+9\alpha} - C (h^{\gamma\beta} + h^{\beta/4}) h - C  h^{1+ \alpha_0/4} \ge \frac12 h^{1+9\alpha},
	\end{split}
	\end{equation}
where we require $9\alpha < \min\{\gamma\beta, \alpha_0/4\}$.   

If $y \in \p \tilde D_h \cap \{y_2 = h\}$ with $\tilde\rho(y_1) \le h^{1+ 8\alpha}$, 
by \eqref{y2-r3}, we have 
	\begin{equation*}
		y_1^2 \leq \frac{-z_1^2+\sqrt{z_1^4+16h^{1+8\alpha}}}{2} = \frac{8h^{1+8\alpha}}{z_1^2+\sqrt{z_1^4+16h^{1+8\alpha}}}.
	\end{equation*}
Denote $t_m = \min \{|t^+|, |t^-|\}$. Again by \eqref{y2-r3}, we get
	\begin{equation*}
		t_m^2 \geq \frac{-z_1^2+\sqrt{z_1^4+h/3}}{2}= \frac{h/6}{z_1^2+\sqrt{z_1^4+h/3}}. 
	\end{equation*}
Hence,
	\begin{equation} \label{y1tm}
		y_1^2 \le \frac{8h^{1+4\alpha}}{z_1^2h^{-4\alpha}+\sqrt{z_1^4h^{-8\alpha}+16h}} \leq 48 h^{4\alpha} t_m^2,
	\end{equation}
which implies that $|y_1| \leq 7 h^{2\alpha} t_m$.  
 
By the convexity of $u$, \eqref{barphi}, and \eqref{t-h-1}, we have
	\begin{equation*} 
		|\p_{y_1} u(y)| \le \frac{\sup_{\tilde D_h} u}{t_m- |y_1|}  \le \frac{Ch}{t_m} 
	\end{equation*}
for $y \in \p \tilde D_h \cap \{y_2 = h\}$ with $\tilde\rho(y_1) \le h^{1+ 8\alpha}$.	
Hence by \eqref{y1tm} and the contradiction assumption $u(he_2)\geq h^{1+\alpha}$, we obtain
	\begin{equation}\label{5.2-}
		u(y_1,h) \ge u(he_2) - |\p_{y_1} u | |y_1| \ge h^{1+\alpha}- \frac{C h }{t_m}  h^{2\alpha}t_m \ge \frac12 h^{1+\alpha}.
	\end{equation}
Therefore, it follows from \eqref{uII1} and \eqref{5.2-} that
	\begin{equation}\label{5.3-}
		u(y) \ge \frac12 h^{1+9\alpha} = \frac12 h^{9\alpha} y_2 \quad\text{ on }\ \p \tilde D_h \cap \{y_2 =h\}.
	\end{equation}

Combining \eqref{u22p} and \eqref{5.3-}, we conclude
	\begin{equation*}
		u(y) \geq \frac12 h^{9\alpha} y_2   \quad\text {on }\  \p \tilde D_h,
	\end{equation*}
provided $\alpha < \frac19\min\{\gamma\beta, \alpha_0/4\}$. 
Choosing the auxiliary function $w$ as in Case $I_{ii}$ 
with $q=1+ \frac{\alpha_0}4$,  
and applying the comparison principle to $w$ and $u-\frac12h^{9\alpha} y_2$ in $\tilde D_h$, we obtain 
	\begin{equation*}
		\partial_{y_2}u(0)  \ge  \frac12 h^{9\alpha} - \frac{2q}{q-1} h^{q-1} > 0
	\end{equation*}
for sufficiently small $h>0$.
This contradicts the assumption that $Du(0)=0$.
Therefore, this lemma is proved.
\end{proof}

\begin{proof}[Proof of Theorem \ref{thm5.1}]
By combining Lemmas \ref{lemma5.2} through \ref{lemma5.4}, we establish \eqref{bdryC1a}.
\end{proof}

\vskip5pt
\subsection{Global $C^{1,\alpha}$ estimate}

Next, we prove, for any $y, w \in \overline\Omega$ close to the boundary $\p\Omega$,
	\begin{equation}\label{e-5.5}
		|Du(y)-Du(w)| \leq C|y-w|^\alpha
	\end{equation}
for some $\alpha\in(0,1)$, where $C>0$ is a constant independent of $y, w\in\overline\Omega$. 
Note that we may assume that $y, w \in B_r(0)\cap\overline\Omega$ are close to the origin for a small $r>0$;   
otherwise, the boundary $\p\Omega$ would be uniformly convex in the vicinity. 

Let $z=(y_1,\rho(y_1))\in\p\Omega$ be the projection of $y$ to $\p\Omega$. 
Applying the coordinates change $\mathcal{T}$ from \eqref{y=tx} and \eqref{y=tx-1}, we assume that $y$ lies on the $e_2$-axis and the boundary $\p\tilde\Omega$ is defined by $\tilde\rho$ as in \eqref{y2-r0}, satisfying \eqref{y2-r3}.
(For simplicity, we continue to use $u$, $y$, and $w$ after this transformation.) 

Let $w_0=(w_1, \tilde\rho(w_1))\in\p\tilde\Omega$ be the projection of $w$ to $\p\tilde\Omega$.
Then,
	\begin{equation*}
		|Du(y)- Du(w)| \le |Du(y)- Du(0)| + |Du(0)- Du(w_0)| + |Du(w_0)- Du(w)|.
	\end{equation*}
Using Theorem \ref{thm5.1}, we get 
	\begin{equation*}
	\begin{split}
		|Du(y)- Du(w)| &\le C \max\{d_y^\alpha, |w_0|^\alpha, d_w^\alpha\} \\
				& \le C \max\{d_y^\alpha, |y-w|^\alpha, d_w^\alpha\}.
	\end{split}
	\end{equation*}
If $|y-w| \ge \max\{d_y^4,\, d_w^4\}$, then $|Du(y)- Du(w)| \le C  |y-w|^{\alpha/4}$, which implies that \eqref{e-5.5} holds for a different $\alpha$.
Therefore, we assume in the following that
	\begin{equation}\label{rz-}
		|y-w| < \max\{d_y^4,\, d_w^4\}.
	\end{equation}
Note that \eqref{rz-} implies that 
$|d_y - d_w| \le \max\{d_y^4,\, d_w^4\}$, meaning that $y, w$ are \emph{``relatively"} interior comparing to their distances from the boundary.  

\begin{lemma}\label{lem-ghg}
Under the hypotheses of Theorem \ref{thm5.1} and the assumption \eqref{rz-},
the estimate \eqref{e-5.5} holds for some positive constants
$\alpha$ and $C$ depending only on $C_b, f_0, \beta,\Omega$ and $\|\varphi\|_{C^{4,\beta}(\p\Omega)}$.
\end{lemma}

\begin{proof}
Subtracting a linear function, we assume $\{y_3=0\}$ is the support plane of $u$ at $y$, with $u(y)=0$ and $u\geq0$.
To prove \eqref{e-5.5}, by the convexity of $u$ it suffices to prove
	\begin{equation}\label{sr-z}
		u(x) \le C |x-y|^{1+\alpha} \quad \text{ for any } x \in B_{2|y-w|}(y),
	\end{equation}
where $B_r(y)$ denotes the ball in $\mathbb R^2$ of radius $r$, centred at $y$.

Let $S_h = \{x\in \tilde\Omega~|~ u(x) < h\}$ be the sub-level set of $u$ at $y$.

\noindent{\bf Case 1:} $u$ is not strictly convex at $y$, namely there is no $h > 0$ such that $S_h \Subset \tilde\Omega$.

In this case, all extreme points of the contact set $\mathcal C_0 := \{x \in \overline{\tilde\Omega} ~|~ u(x) = 0\}$ 
must be on the boundary $\p\tilde\Omega$ due to the doubling condition \eqref{dc}, (see \cite{C-1991}).
Let $y=\sum_{i=1}^m a_ip_i$, where $a_i>0$, $\sum a_i=1$, and $p_i\in\p\tilde\Omega$ are extreme points of $\mathcal{C}_0$. 
By Theorem \ref{thm5.1}, we have $[u]_{C^{1,\alpha}\{p_i\}} \le  C$ for all $i=1,\cdots,m\leq 3$.
For a point $p\in \overline{\tilde\Omega}$, we denote
	\begin{equation*}
		[u]_{C^{1,\alpha}\{p\}} = \inf \big\{ C \ | \ u(x)-\ell_p(x) \leq C|x-p|^{1+\alpha} \ \ \text{for any }\,x\in\tilde\Omega\big\},
	\end{equation*}
where $\ell_p$ is the support plane of $u$ at $p$. See \cite{GT} for other equivalent notations. 
Consequently, by the convexity of $u$, we have $[u]_{C^{1,\alpha}\{y\}} \le  C$, thus proving \eqref{sr-z}.

\noindent{\bf Case 2:} $u$ is strictly convex at $y$, namely $h_0:=\sup \{h~|~ S_h \Subset \tilde\Omega \}>0$.

By the gradient estimate \eqref{C1bound} and the convexity of $u$, we have
	\begin{equation*}
		\dist(y, \partial S_{h_0}) \geq \frac{h_0}{\sup |Du| } \geq ch_0.
	\end{equation*}
If $h_0 > |y-w|^{1-\vep}$ for a small constant $\vep>0$, then $B_{2|y-w|}(y)\Subset S_{h_0}$ since $|y-w|$ is sufficiently small. 
By \eqref{inde}, we can obtain \eqref{sr-z} that for $x\in B_{2|y-w|}(y)$,
	\begin{equation*}
	\begin{split}
		0\le u(x) &\le \frac{2 h_0}{ \big(\text{dist}(y,\p S_{h_0})\big)^{1+\alpha}} |x-y|^{1+\alpha}  \\
			& \le \frac{C}{ h_0^\alpha } |x-y|^{1+\alpha}  \le C |x-y|^{1+\alpha\vep},
	\end{split}
	\end{equation*}
where the last inequality is due to the assumption $h_0>|y-w|^{1-\vep}>(|x-y|/2)^{1-\vep}$.  

Assume now that $h_0 \le |y-w|^{1-\vep}$.
Let $p \in \p \tilde\Omega \cap \p S_{h_0}$, and $p^*\in\p S_{h_0}$ such that $y \in \overline{pp^*}$, where $\overline{pp^*}$ is the line segment connecting $p$ and $p^*$.
Define $e_p := \frac{p-y}{|p-y|}$.
Consider the following convex cone $\Sigma$, introduced in \cite{CTW-2022}, defined as
	\begin{equation}\label{coner}
		\Sigma = \big\{x\in \R^2 ~|~ x = p^* + t(\tilde z - p^*),\ t>0, \ \tilde z \in \p \tilde\Omega\cap B_r(p) \big\},
	\end{equation}
where $r>0$ is a constant. By \eqref{sigma}, we can choose $r=\frac{2}{\sigma}|y-w|^{3/4}$ so that $w$ is contained in the cone $\Sigma$.  

Define a function $v$ by
	\begin{equation*}
		v(x) = u(p^*) + t(u(\tilde z) - u(p^*)) \quad \text{ for }\ x = p^* + t(\tilde z - p^*)\in \Sigma,\ \tilde z \in \p \tilde\Omega.
	\end{equation*}
The graph of $v$ is a two-dimensional cone in $\R^3$ with vertex at $(p^*, u(p^*))$ and $v=\tilde\varphi$ on $\p \tilde\Omega$ near $p$.
By the convexity of $u$, we have
	\begin{equation*}
		u(x) \leq v(x) \quad \text{ for any } x\in\Sigma\cap \tilde\Omega.
	\end{equation*}

 Let $\tau$ be the unit tangential vector of $\p \tilde\Omega$ at $p$ such that $\theta_{\tau,e_p}$, the angle between the vectors $\tau$ and $e_p$, lies in $(0,\pi/2)$.
We claim that there is a lower bound of $\theta_{\tau,e_p}$ that
	\begin{equation}\label{lbtheta}
		\theta_{\tau,e_p} \geq \frac{1}{144}|y|^{\frac34}.
	\end{equation}
Since $p \in \p \tilde\Omega \cap \p S_{h_0}$, by convexity we have $\tilde\varphi(p)=h_0$ and $\tilde\varphi(x) \ge h_0$ for $x\in\p\tilde\Omega$ near $p$. 
Hence by the Taylor expansion, we get
	\begin{equation*}\label{u-p-2}
		v(x) \le h_0 + C |x-p|^2 \quad \text{ for any }  x\in\p \tilde\Omega  \text{~near~} p.
	\end{equation*}
Assuming \eqref{lbtheta} for the moment, we then have
	\begin{equation*}\label{v1est}
		v(x) \leq h_0 + C\left(\frac{|x-p|}{|y|^{3/4}}\right)^2 \quad \text{ for } x\in\Gamma_1, 
	\end{equation*}
where 
	\begin{equation*}
		\Gamma_s := \left\{x \in \Sigma ~\Big|~ \frac{x-p^*}{|p-p^*|}\cdot e_p = s \right\},\quad s\in(0,1].
	\end{equation*}
By the assumption \eqref{rz-} and the choice of $r$ in \eqref{coner}, we get $|y|^{\frac32}>|y-w|^{\frac38}>(\frac{\sigma}{2}|x-p|)^{\frac12}$ for $x\in\Gamma_1$.
Hence, we obtain
	\begin{equation}\label{v-h1}
		v(x) \leq h_0 + C|x-p|^{3/2} \quad\text{ for any } x\in\Gamma_1. 
	\end{equation}
By the definition of $v$, \eqref{sigma} and \eqref{v-h1}, we then have
	\begin{equation}\label{v-h2}
		v(x) \leq h_0 + C \sigma^{-1} |x-y|^{3/2} \quad\text{ for any } x\in\Gamma_{s^*},
	\end{equation}
where $s^* = \frac{|y-p^*|}{|p-p^*|} \ge \frac{\sigma}{1+\sigma}$ such that $y\in \Gamma_{s^*}$.

Since $v\equiv h_0$ on the line segment $\overline{pp^*}$, we obtain 
	\begin{equation}\label{v-h3}
		v(x) \leq h_0 + C |x-y|^{3/2} \quad \text{ for any } x \in B_{4|y-w|^{3/4}}(y)
	\end{equation}
for a different constant $C$ depending on $\sigma$. 
To see \eqref{v-h3}, for any $x\in B_{4|y-w|^{3/4}}(y)$, there exists $s'\in(0,1)$ such that $x\in\Gamma_{s'}$. 
By \eqref{rz-}, $|y-w|^{3/4} \ll d_y^3 \ll |p-p^*|$. Thus, $s'$ is almost equal to $s^*$, i.e. $|s'-s^*|\ll 1$. 
Let $x'=\Gamma_{s'}\cap\overline{pp^*}$. Similar to \eqref{v-h2}, we have $v(x) \leq h_0 + C' |x-x'|^{3/2} \leq h_0 + C' |x-y|^{3/2}$ for a difference constant $C'$. 

Now, we are ready to prove the regularity of $u$ in \eqref{sr-z} by using that of $v$ in \eqref{v-h3}.
For $x\in B_{2|y-w|^{3/4}}(y)$, we consider two cases: (i) $u(x)\geq h_0$ and (ii) $u(x)<h_0$.

In case (i), let $\hat x = y+ 2(x-y)$.
By the convexity of $u$ and applying \eqref{v-h3} at $\hat x$, we get
	\begin{equation*}
		2h_0 \leq 2u(x) \leq u(\hat x) \leq v(\hat x) \leq h_0 + C|x-y|^{3/2},
	\end{equation*}
which yields $h_0\leq C|x-y|^{3/2}$.	
Then, by applying \eqref{v-h3} at $x$, we obtain
	\begin{equation}\label{580}
		u(x) \le v(x) \le h_0 + C|x-y|^{3/2} \leq C|x-y|^{3/2}.
	\end{equation}

In case (ii), there exists a unique $\tilde x=y+\tilde t(x-y)\in\partial S_{h_0}$ for some $\tilde t>1$ such that $u(\tilde x)= h_0$.
If $|\tilde x -y| \le 2|y-w|^{3/4}$, then by \eqref{inde} with $\alpha\leq1/2$, we have
	\begin{equation*}
		u(x) \le |x-y|^{1+\alpha} \frac{2u(\tilde x)}{|\tilde x-y|^{1+\alpha}} \le C |x-y|^{1+\alpha},
	\end{equation*}
where the second inequality follows from \eqref{580}, applied at $\tilde x$.
If $|\tilde x -y| > 2|y-w|^{3/4}$, then by the assumption $h_0\leq |y-w|^{1-\varepsilon}$ and \eqref{inde} we have
	\begin{equation*}
	\begin{split}
		u(x) & \le |x-y|^{1+\alpha} \frac{2h_0}{|\tilde x-y|^{1+\alpha}} \\
			& \le  C |x-y|^{1+\alpha}  |y-w|^{\frac{1-4\varepsilon-3\alpha}{4}} \le C |x-y|^{1+\alpha} 
	\end{split}
	\end{equation*}
for some $\alpha>0$ small such that $1-4\varepsilon-3\alpha\geq0$. 
Therefore, the desired estimate \eqref{sr-z} is proved.

Finally, it remains to prove the claim \eqref{lbtheta}.
We may assume that $y=(0,y_2)$ is on the positive $y_2$-axis and $p=(p_1,p_2)$ satisfies $p_1 > 0$.
Recall that $\tau$ is the unit tangential vector at $p\in\partial \tilde\Omega$ and $e_p=\frac{p-y}{|p-y|}$. 
We shall derive the lower bound of the angle $\theta_{\tau,e_p}$ between $\tau$ and $e_p$. 

If $p_2>y_2$, then $\theta_{\tau,e_p}=\theta_{\tau,e_1}-\theta_{e_1,e_p}$, and 
	\begin{equation}\label{the1}
	\begin{split}
		\theta_{e_1,e_p} &= \arccos(e_1\cdot e_p)  = \arctan \frac{p_2-y_2}{p_1} \\ 
					&\approx \frac{p_2-y_2}{p_1}  \le \frac{p_2}{p_1}. 
	\end{split}
	\end{equation}
By \eqref{y2-r0}, we have $p_2 =\tilde \rho(p_1)= 6 |z_1|^2 p_1^2 + 4 z_1 p_1^3 + p_1^4$, and $\tau = \frac{(1, \tilde\rho'(p_1))}{\sqrt{1+(\tilde \rho'(p_1))^2}}$. 
 Then, 
	\begin{equation}\label{the2}
	\begin{split}
		\theta_{\tau, e_1} &= \arccos(\tau\cdot e_1) = \arctan \tilde\rho'(p_1) \\
				& \approx \tilde\rho'(p_1).
 	\end{split}
	\end{equation}
Hence, from \eqref{y2-r0} and \eqref{y2-r3} we obtain that
	\begin{equation}\label{theta-tp1}
	\begin{split}
		\theta_{\tau, e_p}  &= \theta_{\tau, e_1} - \theta_{e_1,e_p} \ge \tilde \rho'(p_1) - \frac{p_2}{p_1}   \\
 				& \ge \frac1{6p_1} \big( |z_1|^2 p_1^2 + p_1^4 \big) \ge \frac{p_2}{72p_1}  \\
			& \ge \frac{1}{144} p_2^{\frac34} \ge \frac{1}{144} y_2^{\frac34} = \frac{1}{144} |y|^{\frac34},
	\end{split}
	\end{equation}
where the last inequality is due to the assumption $p_2>y_2=|y|$.

If $p_2 \le y_2$, then $\theta_{\tau, e_p}=\theta_{\tau, e_1} + \theta_{e_1,e_p}$.
In this case, \eqref{the2} still holds, while \eqref{the1} becomes 
	\begin{equation*}
		\theta_{e_1,e_p} = \arctan \frac{y_2-p_2}{p_1} \approx \frac{y_2-p_2}{p_1}
	\end{equation*}
if $\frac{y_2-p_2}{p_1} \ll 1$  as $|y|\to 0$.
(Note that if $\frac{y_2-p_2}{p_1} \geq c_0$, then $\theta_{e_1,e_p}\geq\arctan c_0$, which immediately provides the required bound.)
Hence, by \eqref{y2-r0} and \eqref{y2-r3}, we have 
	\begin{equation}\label{theta-tp2}
	\begin{split}
		\theta_{\tau, e_p}  &= \theta_{\tau, e_1} + \theta_{e_1,e_p} \approx \frac{y_2}{p_1} +  \tilde\rho'(p_1) -\frac{p_2}{p_1}  \\
					& \ge \frac{y_2}{p_1} + \frac{p_2}{72p_1} \ge 4 \Big( \big(\frac{y_2}{3p_1} \big)^3 \frac{p_2}{72p_1}\Big)^{1/4} \ge \frac{1}{6} y_2^{\frac34} = \frac{1}{6} |y|^{\frac34}.
	\end{split}
	\end{equation} 
Combining \eqref{theta-tp1} and \eqref{theta-tp2}, we establish the claim \eqref{lbtheta}, completing the proof of Lemma \ref{lem-ghg}.

\end{proof}

\vspace{5pt}
\section{Proof of Theorem \ref{mt-1}}\label{S4}

Assume $0\in\p\Omega$ is an isolated degenerate point, namely $\kappa(0)=0$ while $\kappa>0$ on $\p\Omega$ near the origin, where $\kappa$ is the curvature of $\p\Omega$.
Choosing appropriate coordinates, assume $\Omega \subset \{x_2 >0\}$ and that $\p\Omega$ is locally described as in \eqref{x2-rho-1} with \eqref{x2-rho-2}. 
Suppose $\p\Omega\in C^{k+\beta}$, $\beta\in(0,1]$, is $k$-order degenerate at $0$.
In Section \ref{S3}, we examined the model case $k=4$; here, we consider the general case for $k\geq4$ being an even integer.

\begin{lemma}\label{lem-5.1}
Let $u$ be a convex function on $\overline\Omega$ and $\varphi = u|_{\p\Omega}\in C^{k,\beta}(\p\Omega)$ satisfy $(\mathcal P1)$.
Then, by subtracting any affine function from $u$, the condition $(\mathcal P1)$ remains invariant, implying that 
\begin{itemize}
  \item[$(\mathcal P1')$] $\frac{\p^i \varphi}{\p x_1^i}(0) = 0\ $ for all $i=2,3, \cdots, k-1$.
\end{itemize} 
\end{lemma}

\begin{proof}
By subtracting the support function of $u$ at $0$, we may assume that $\varphi(0)=0$, $\varphi'(0)=0$ and $\varphi\geq0$. 
It is easy to see that $(\mathcal P1)$ implies $(\mathcal P1')$, since any non-zero odd-order derivative would contradict the non-negativity of $\varphi$. 

Consider an affine function $\ell=a_0+a_1x_1+a_2x_2$ with constants $a_0,a_1,a_2$. 
Let $\hat u=u-\ell$ and $\hat\varphi=\hat u|_{\p\Omega}$. We verify that $\hat\varphi$ satisfies $(\mathcal P1)$ and $(\mathcal P1')$ as well. 

Since $\p\Omega\in C^{k+\beta}$ is $k$-order degenerate at point $0$, \eqref{k-deg} and \eqref{k+b-deg} yield
	\begin{equation}\label{kbdr}
		x_2=\rho(x_1) = \frac{1}{k!} \rho^{(k)}(0) x_1^k + \bar\rho(x_1),
	\end{equation}
where $\rho^{(k)}(0) \neq 0$ and $\bar\rho(x_1) = O\big(|x_1|^{k+\beta}\big)$. 
Hence,
	\begin{equation*}
		\hat\varphi(x_1) = \varphi(x_1) - \ell|_{\p\Omega} = \varphi(x_1) - \left(a_0+a_1x_1+a_2\Big(\frac{1}{k!} \rho^{(k)}(0) x_1^k + \bar\rho(x_1)\Big)\right).
	\end{equation*}
It is clear that $	\hat \varphi^{(i)}(0) = \varphi^{(i)}(0)$ for all $i= 2, 3, \cdots, k-1.$
Therefore, $\hat\varphi$ satisfies both $(\mathcal P1)$ and $(\mathcal P1')$.
\end{proof}

\begin{remark}\label{re-5.1}
If $\p\Omega\in C^{k+\beta}$, with $k\ge 4$ even and $\beta\in(0,k]$, is $(k,\beta)$-order degenerate at $0$, and $\varphi \in C^{k+\beta}(\p\Omega)$ satisfies $(\mathcal P2)$,
then $(\mathcal P2)$ remains invariant under subtraction of an affine function from $u$.
\end{remark}

The key components of the proof are the local coordinate transformation $\mathcal T$ (analogous to \eqref{y=tx} and \eqref{y=tx-1}) and the properties of boundary data $\tilde\varphi$ in the new coordinates, as established in \eqref{var2-2}--\eqref{g-v-up}. 

\begin{lemma}\label{lem-boundary}
Assume that $\varphi\in C^{k,\beta}(\p\Omega)$ satisfies $(\mathcal P1)$.
For $z\in\p\Omega$ near the origin, there exists a transform $y = \mathcal T_k (x)$ (defined in \eqref{t-k}) such that $\mathcal T_k(z)=0_y$, $\tilde\Omega:= \mathcal T_k(\Omega) \subset \{y_2>0\}$, $\det D\mathcal T_k = 1$, and locally $\p\tilde\Omega = \{y_2 = \tilde\rho(y_1)\}$ with $\tilde\rho(0_y)=\tilde\rho'(0_y)=0$.
Moreover, for all $y\in\p\tilde\Omega$ near $0_y$, 
\begin{itemize}
  \item[$(i)$] $\tilde\rho(y_1) \approx \kappa(z) y_1^2 + y_1^k$,
  \item[$(ii)$] $y_1\tilde\rho'(y_1) \approx \kappa(z) y_1^2 + y_1^k$,
  \item[$(iii)$] $y_1\tilde\rho'(y_1) - \tilde\rho(y_1)  \approx \kappa(z) y_1^2 + y_1^k$,
\end{itemize}
where $a \approx b$ denotes $C_k^{-1}b \leq a \leq C_k b$ for a constant $C_k>0$ depending only on $k$.
Additionally, the Taylor expansion of $\tilde\varphi(y)$ at $0_y$ satisfies
\begin{itemize}
  \item[$(iv)$] $ \tilde\varphi(y) - \tilde\varphi(0_y) - \tilde\varphi'(0_y)y_1  = \tilde\varphi^{(k)}(0_y) E_1 + E_2,$
\end{itemize}
where $E_1 \approx y_2$ and $E_2 \approx O \big(\kappa(z)^{\frac{\beta}{k-2}} + |y_2|^{\frac{\beta}{k}} \big) y_2$. 
\end{lemma}

\begin{proof}
By the assumption of $\p\Omega$, we have \eqref{kbdr} locally. By scaling, we may assume that $\rho^{(k)}(0) = k!$, so $\rho(x_1) =  x_1^k + \bar\rho(x_1)$ for some $\bar\rho$ satisfying $\bar\rho(x_1)=O(|x_1|^{k+\beta})$. 

The tangent line of $\p\Omega$ at $z$ is then $l_{z}(x_1) = \big(kz_1^{k-1} +\bar\rho'(z_1)\big) (x_1 - z_1) + z_1^k + \bar\rho(z_1)$. 
Using the coordinate transform 
	\begin{equation}\label{t-k}
		\mathcal T_k:
		\begin{cases}
			y_1 = x_1 - z_1 \\
			y_2 = x_2 - l_{z}(x_{1}) 
		\end{cases}
	\end{equation}
yields $\det D\mathcal T_k =1$, with $\mathcal T_k$ mapping $z$ to $0_y$, the origin in the $(y_1,y_2)$-coordinate.
Similar to \eqref{y2-r0}, $\tilde\Omega = \mathcal T_k (\Omega) \subset \{y_2 >0\}$, and $\p \tilde\Omega$ is locally described by
	\begin{equation*} 
	\begin{split}
		y_2 = \tilde\rho(y_1) &= \Big( (y_1 + z_1)^k - kz_1^{k-1} y_1 - z_1^k \Big) + \Big( \bar\rho(y_1 + z_1) -\bar\rho'(z_1) y_1 - \bar\rho(z_1) \Big) \\
				&=: I_1+I_2.
	\end{split}
	\end{equation*}
By Lemma \ref{p-k}, we have $I_1 \approx |z_1|^{k-2} y_1^2 + |y_1|^k$.
By the Taylor expansion of $\bar\rho$ at $z_1$, $\bar\rho(x_1)=O(|x_1|^{k+\beta})$, we obtain
	\begin{equation*}
	\begin{split}
 		I_2 &=  \frac{1}{2!} \bar\rho^{(2)}(z_1) y_1^2 + \cdots + \frac{1}{k!} \bar\rho^{(k)}(z_1) y_1^k + O (|y_1|^{k+\beta}) \\
			& = O \Big( |z_1|^{\beta} \big( \sum_{i=2}^k |z_1|^{k-i} |y_1|^i \big) \Big) +  O (|y_1|^{k+\beta}) = o(I_1).
	\end{split}
	\end{equation*}
Hence, analogous to \eqref{y2-r3}, we conclude
	\begin{equation}\label{appy2}
		y_2 = \tilde\rho(y_1) \approx |z_1|^{k-2} y_1^2 + |y_1|^k.
	\end{equation}

The curvature of $\p\Omega$ at $z$ satisfies
	\begin{equation}\label{appkappa}
		\kappa(z) \approx \rho''(z_1) = k(k-1) z_1^{k-2} + \bar\rho''(z_1) \approx z_1^{k-2},
	\end{equation}
implying conclusion $(i)$.
Conclusion $(iii)$ follows from Lemma \ref{q-k}, as presented in the appendix. Conclusion $(ii)$ can then be deduced by $(i)$ and $(iii)$.

\vskip5pt
Next, we investigate the boundary function $\tilde\varphi(y)=\varphi(\mathcal T_k^{-1}y)$ for $y\in\partial \tilde\Omega$. 
We extend $\tilde\varphi$ from $\p\tilde\Omega$ to $\tilde D_h = \{y\in \tilde\Omega ~|~ \tilde \rho(y_1) < y_2 < h\}$ for small $h>0$ by
	\begin{equation}\label{extphi1}
		\tilde\varphi(y_1,y_2) = \tilde\varphi(y_1, \tilde\rho(y_1)),\quad (y_1,y_2)\in \tilde D_h. 
	\end{equation}  
By condition $(\mathcal P1')$, we get $\tilde \varphi^{(i)}(-z_1) =0$, $i=2, 3 \cdots, k-1$.
Using Taylor's expansion and the $C^{k,\beta}$ regularity of $\tilde\varphi$, analogous to \eqref{var-2-3} we have
	\begin{equation*}\label{k-var-2-3} 
		\tilde\varphi^{(k-i)}(0) = \frac{1}{i!}z_1^{i}\tilde\varphi^{(k)}(0) + O(|z_1|^{i+\beta}),\quad i=1,2,\cdots,k-2.   
	\end{equation*}
Therefore, analogous to \eqref{taylor-1}, we can express the Taylor expansion as follows:
	\begin{align*}\label{k-taylor-1} 
		\tilde\varphi(y_1, \tilde\rho(y_1)) &= \tilde\varphi(0) + \tilde\varphi'(0)y_1 + \sum_{i=2}^k \frac1{i!} \tilde\varphi^{(i)}(0)y_1^i + O(|y_1|^{k+\beta}) \\
				& = \tilde\varphi(0) + \tilde\varphi'(0)y_1 + \tilde\varphi^{(k)}(0) \Big( \sum_{i=2}^k \frac1{i!(k-i)!} y_1^i z_1^{k-i}  \Big) \\
				&\qquad + O \Big( \sum_{i=2}^k |z_1|^{k-i+\beta} |y_1|^i \Big)+ O(|y_1|^{k+\beta}) \\
				& =:  \tilde\varphi(0) + \tilde\varphi'(0)y_1 +  \tilde\varphi^{(k)}(0) E_1 + E_2. 
	\end{align*}
From \eqref{appy2} and by computation, we have
	\begin{align*} 
		& E_1 = \sum_{i=2}^k \frac1{i!(k-i)!} y_1^i z_1^{k-i}  = \frac{1}{k!}\Big((y_1 + z_1)^k- z_1^k -k y_1z_1^{k-1} \Big) \\
		 	&\qquad  \approx |z_1|^{k-2} |y_1|^2 + |y_1|^k \approx y_2, \\
		& E_2 = O\Big( \sum_{i=2}^k |z_1|^{k-i+\beta} |y_1|^i \Big)+ O(|y_1|^{k+\beta}) = O (|z_1|^{\beta} + |y_1|^{\beta}) y_2.
	\end{align*}
Hence, conclusion $(iv)$ follows from \eqref{appy2} and \eqref{appkappa}.
\end{proof}

After obtaining the estimates of boundary data in the new $y$-coordinates, we can modify the auxiliary function $w$ in \eqref{w-un-0} to be
	\begin{equation}\label{neww}
		w(y) = \frac12y_1^2y_2^{q-\frac2k} + Qy_2^q - My_2 \quad\text{ in } \ \tilde D_h=\tilde\Omega\cap\{y_2<h\}.
	\end{equation}
By choosing $q\in(1,1+\frac1k)$, $Q=\frac{k}{q-1}(q+1-\frac4k)$, $h\leq (\frac{kf_0}{2k-4})^{1/(2q-\frac2k-2)}$ and $M=(k+Q)h^{q-1}$, we can verify \eqref{dw-0} and \eqref{w-0}. Hence, we can apply \eqref{neww} as the barrier function to derive the gradient estimate as stated in Lemma \ref{thm-gra-est} and the H\"older continuity for the normal derivative as in Lemmas \ref{lemma5.3} and \ref{lemma5.4}. 

By Remark \ref{rmk3.1}, we can then obtain the boundary $C^{1,\alpha}$ estimate \eqref{bdryC1a} as in Theorem \ref{thm5.1} for some exponent $\alpha\in(0,\beta/k)$ and constant $C>0$ depending on $C_b, f_0, k, \beta, \Omega$ and $\|\varphi\|_{C^{k,\beta}(\p\Omega)}$. 

In deriving the global $C^{1,\alpha}$ estimate \eqref{e-5.5}, a crucial ingredient, following the lines of Lemma \ref{lem-ghg}, is the lower bound of the angle $\theta_{\tau,e_p}$ as in \eqref{lbtheta}.
For the general case of $k\geq4$, we derive this estimate in the following lemma. 

\begin{lemma}\label{angle-k} 
Assume that $y=(0,y_2)\in  \tilde\Omega$ such that $|y|=y_2>0$ is small and $\{y_3=0\}$ is the support plane of $u$ at $y$.
Let $p\in\p\tilde\Omega \cap S_{h_0}$, where $h_0=\sup\{h | S_h\Subset \tilde\Omega\}>0$.
Then,  
	\begin{equation*}
		\theta_{\tau,e_p} \ge C_k |y|^{\frac{k-1}{k}},
	\end{equation*}
where $e_p = \frac{p-y}{|p-y|}$, $\tau$ is the unit tangential vector of $\p \tilde\Omega$ at $p$, $\theta_{\tau,e_p}\in(0,\pi/2)$ is the angle between the vectors $\tau$ and $e_p$, and the constant $C_k>0$ depends only on $k$.  
\end{lemma}

\begin{proof}
Recall that the new $y$-coordinates in Lemma \ref{lem-boundary} are based at a boundary point $z\in\p\Omega$ in the original $x$-coordinates. 
If $z\in \{x_2 \geq \delta_1\}$ for a constant $\delta_1>0$, then $\kappa(z) \ge \delta_2$ for some $\delta_2>0$, implying that the boundary is locally uniformly convex. 
Therefore, it suffices to consider the case where $z\in \{x_2 < \delta_1\}$ close to the origin.
We will then apply the coordinate change \eqref{t-k} and estimate the angle $\theta_{\tau,e_p}$ in the $y$-coordinates. 

Similar to the proof of \eqref{lbtheta} in Lemma \ref{lem-ghg}, if $p_2>y_2$, then by \eqref{the1} and \eqref{the2}, we have
	\begin{equation*}
		\theta_{\tau,e_p} = \theta_{\tau,e_1} - \theta_{e_1,e_p} \geq \frac{1}{p_1}\big(p_1\tilde\rho'(p_1) - p_2\big).
	\end{equation*}
Then, by $(i)$ and $(iii)$ of Lemma \ref{lem-boundary}, we obtain
	\begin{equation*}
		\theta_{\tau,e_p} \geq C_k\frac{p_2}{p_1} \geq C_kp_2^{\frac{k-1}{k}} = C_k|y|^{\frac{k-1}{k}}.
	\end{equation*}
	
Similar to \eqref{theta-tp2}, if $p_2\leq y_2$, by $(i)$ and $(iii)$ of Lemma \ref{lem-boundary} again, we have 
	\begin{equation*}
	\begin{split}
		\theta_{\tau,e_p} &\geq \frac{y_2}{p_1} + C_k\frac{p_2}{p_1} \geq C_k\Big( (\frac{y_2}{p_1})^{k-1} \frac{p_2}{p_1} \Big)^{1/k} \\
			& \geq C_ky_2^{\frac{k-1}{k}} = C_k|y|^{\frac{k-1}{k}},
	\end{split}
	\end{equation*}
where the constant $C_k$ may vary from line to line but depends only on $k$. 
Hence, the proof of Lemma \ref{angle-k} is complete. 
\end{proof}

\begin{proof}[Proof of Theorem \ref{mt-1}]
Once we have the key ingredients, namely Lemmas \ref{lem-boundary} and \ref{angle-k}, we can obtain the global $C^{1,\alpha}$ estimate 
following the lines of Lemma \ref{lem-ghg}. Together with the boundary $C^{1,\alpha}$ estimates mentioned before Lemma \ref{angle-k}, we can complete the proof of 
Theorem \ref{mt-1}.
\end{proof}

\begin{remark}
\emph{
It is worth mentioning that the quadratic separation condition \eqref{u-sq} is not satisfied if $\varphi \equiv 0$ on  $\p\Omega$, where $\p\Omega\in C^{k+\beta}$ is $k$-order degenerate at $0$ for an even number $k\ge 4$.
However, $(\mathcal P1)$ is satisfied in this case. 
Then, we have the following corollary.
\begin{corollary}
Let $u$ be a convex solution of \eqref{u-fp} with $\varphi\equiv 0$ and $f$ satisfying \eqref{dc} and \eqref{f0}.
Assume that $\Omega$ is a bounded, smooth, and strictly convex domain in $\R^2$, and $\p\Omega$ is $k$-order degenerate at each isolated degenerate point $z\in\p\Omega$ for some even number $k=k(z)\ge 4$.
Then, $u$ is $C^{1,\alpha}(\overline\Omega)$ for some $\alpha \in (0,1)$, and
	\begin{equation*}
		\|u\|_{C^{1,\alpha}(\overline\Omega)} \le C,
	\end{equation*}
where $\alpha, C>0$ are constants depending only on $C_b, f_0$ and $\Omega$. 
\end{corollary}
}
\end{remark}

\vskip5pt
\section{Homogeneous Monge-Amp\`ere Equation}
 \label{S5}

In this section, we prove the global $C^{1,\alpha}$ regularity for the Dirichlet problem
	\begin{equation}\label{u-f=0.1}
	\begin{split}
		\det D^2u &=0 \quad \, \text{in} \ \ \Omega,\\  
		u &= \varphi \quad\text{on} \ \ \p\Omega,
	\end{split}
	\end{equation}
where $\Omega$ is a bounded convex domain in $\R^2$ and $\p\Omega$ is locally described as in \eqref{x2-rho-1} and \eqref{x2-rho-2}. 
We assume that $\p\Omega\in C^{k+\beta}$, $\beta\in(0,k]$, is $(k,\beta)$-order degenerate at $0$, where $k\geq4$ is an even integer. 
Additionally, we assume that $\varphi\in C^{k+\beta}(\p\Omega)$ satisfies conditions $(\mathcal P1)$ and $(\mathcal P2)$.

From Theorem \ref{lem-3.2} (see the proof of \cite[Lemma 3.2]{CTW-2022}), we have the interior regularity $[u]_{C^{1,1}\{x_0\}} \le C$, provided $\p\Omega, \varphi \in C^{1,1}$, where the constant $C$ depends on $\dist(x_0,\p\Omega)$. 

The global regularity $u \in C^{1,1}(\overline\Omega)$ was obtained in \cite{CNS-1986,GTW-1999} under the assumption that $\Omega$ is uniformly convex and $\p\Omega, \varphi \in C^{3,1}$.
More recently, Caffarelli, Tang and Wang \cite{CTW-2022} proved that $u \in C^{1,\beta/2}(\overline\Omega)$ if $\Omega$ is uniformly convex and $\p\Omega, \varphi \in C^{2+\beta}$ for $\beta\in(0,2]$.

In this paper we deal with non-uniformly convex domains $\Omega$ by adopting the local coordinate transform $\mathcal{T}$ in Lemma \ref{lem-boundary}. 
Specifically, from Theorem \ref{mt-1}, we already have that $u \in C^{1,\alpha}(\overline\Omega)$ for a small $\alpha > 0$. 
We aim to improve this exponent $\alpha$ to $\beta/k$ for the solution $u$ of the homogeneous Monge-Amp\`ere equation \eqref{u-f=0.1}, where $\beta\in(0,k]$.

\begin{theorem}\label{thm-2.1}
Assume that $\p\Omega, \varphi \in C^{k+\beta}$ satisfy the hypotheses of Theorem \ref{mt-2}.
Then the solution $u$ of \eqref{u-f=0.1} is $C^{1,\beta/k}(\overline\Omega)$, and
	\begin{equation*}
		\|u\|_{C^{1,\beta/k}(\overline\Omega)} \le C,
	\end{equation*}
where the constant $C>0$ depends only on $k,\beta,\Omega$ and $\|\varphi\|_{C^{k+\beta}(\p\Omega)}$.
\end{theorem}

\begin{proof}
By Theorem \ref{mt-1}, we have the gradient estimate $u \in C^{0,1}(\overline\Omega)$.
Due to the convexity of $u$, it suffices to prove that, for any $x_0\in\overline\Omega$, 
	\begin{equation*}
		\inf\{ C ~|~ u(x)-\ell_{x_0}(x) \leq C|x-x_0|^{1+\beta/k}\ \text{ for any }\,x\in\overline\Omega\} =:[u]_{C^{1,\beta/k}\{x_0\}} \leq C_0,
	\end{equation*}
where $\ell_{x_0}$ is the support of $u$ at $x_0$, and $C_0$ is a positive constant depending only on $k,\beta,\Omega$ and $\|\varphi\|_{C^{k+\beta}(\p\Omega)}$, but independent of $x_0$. 
By \cite[Theorem 2.8]{RT-1977}, the solution $u$ of \eqref{u-f=0.1} can be expressed as
	\begin{equation*}
		u(x) = \sup \{ \ell(x) ~:~ \ell \le \varphi \text{~on~} \p\Omega \text{~and~} \ell \text{~is~affine}\}.
	\end{equation*}
The contact set $\mathcal C_{x_0}=\{x\in\overline\Omega : u(x)=\ell_{x_0}(x)\}$ is a convex set containing a segment, and the extreme points of $\mathcal C_{x_0}$ must lie on the boundary $\p\Omega$.
Again by the convexity of $u$, we just need to prove 
	\begin{equation}\label{6.5}
		[u]_{C^{1,\beta/k}\{y\}} \le C_0 \quad \text{ for any } y\in \p\Omega.
	\end{equation}

We divide the proof of \eqref{6.5} into the following three cases: 

\noindent\item[$(i)$] There is a point $x_0\in\Omega$ with $|x_0-y| \ge \delta_0$ for some fixed small constant $\delta_0>0$, such that $y \in \mathcal C_{x_0}$. Then, there exists some point $\hat x\in\mathcal C_{x_0}\cap\Omega$ satisfying $\dist(\hat x, \p\Omega) \geq C\delta_0^k$ by \eqref{k+b-deg1} when $\delta_0$ is small. Hence, the estimate \eqref{6.5} follows from the proof of Theorem \ref{lem-3.2}, where $C_0$ depends on $\delta_0$. 

\noindent\item[$(ii)$] There is $x_0\in\Omega$ such that $y \in \mathcal C_{x_0}$, but all line segments in $ \mathcal C_{x_0}$ are very short. Note that if there is one segment in $\mathcal C_{x_0}$ with length greater than $\delta_0$, then by the triangle inequality, there must exist a segment in $\mathcal C_{x_0}$ with the endpoint $y$ whose length is greater than $\frac12\delta_0$, which falls under case $(i)$. 

\noindent\item[$(iii)$] There is no point $x_0\in\Omega$ such that $y \in \mathcal C_{x_0}$. Let $x_i \to y$ be a sequence of points in $\Omega$. Choose an extreme point $y_i \in \p\Omega$ of $\mathcal C_{x_i}$ such that $y_i \to y$.  
Thus, if \eqref{6.5} holds at $y_i$ with a uniform constant $C_0$, taking the limit yields that \eqref{6.5} holds at $y$ as desired.

Therefore, we only need to consider case $(ii)$, which is proved in the following lemma.
Noting that the curvature of $\p\Omega$ has a positive lower bound for $y\in\p\Omega \cap\{x_2 \ge \delta_1\}$, it suffices to consider $y\in\p\Omega \cap\{x_2<\delta_1\}$ near the origin. 
\end{proof}

\begin{lemma}\label{lem-6.1}
Assume the conditions in Theorem \ref{thm-2.1}.
Let $\overline{pq}$ be a segment with $p, q \in \p\Omega$ such that $u$ is linear on $\overline{pq}$. 
Then \eqref{6.5} holds at $p$ and $q$. 
\end{lemma}

\begin{proof}
The proof follows the approach in \cite{CTW-2022}, where $\Omega$ is assumed to be uniformly convex. 
Here, we will employ a local coordinate transform $\mathcal{T}$, as defined in Lemma \ref{lem-boundary}, to study the degenerate domain. 
Based on prior discussions, it suffices to consider the case when $|p-q|>0$ is very small and $p,q$ are close to the origin.

Let $z\in\p\Omega$ be a point near the origin such that the tangential vector $\tau_z$ of $\p\Omega$ is parallel to $\overline{pq}$.
We may assume $z_1\geq0$ in the following. If $z_1\leq0$, the proof proceeds in a similar manner.
We apply the coordinate transform $\mathcal T=\mathcal T_k$ defined in \eqref{t-k}, ensuring that $\mathcal T(z) = 0$ and $(i)$--$(iv)$ in Lemma \ref{lem-boundary} hold.
Since $\mathcal T$ is affine, the segment $\overline{pq}$ is parallel to the $e_1$-axis, while we continue to use $p, q$ to denote the points in the $y$-coordinates. We may assume that $p_1 >0$ and $q_1<0$, and will prove that \eqref{6.5} holds at $p$. A similar argument will apply to $q$. 

The main difficulty is to estimate the growth of $u$ in the $e_2$-direction using the Dirichlet boundary condition, given that the tangential vector $\tau_p$ is relatively close to $e_1$. 
Recall that we extended $\varphi$ using \eqref{extphi1} as a function independent of $y_2$, namely
	\begin{equation*}\label{bdphi}
		\varphi(y_1,y_2) = \varphi(y_1,\rho(y_1)) \quad \text{ in }\ \{|y_1|<r_0\}\times\mathbb{R},
	\end{equation*}
where $r_0>0$ is a fixed small constant. We continue to use $\varphi, \rho,\Omega$ and $u$
after the transform $\mathcal T$ for simplicity. 
Note that locally $\p\Omega\cap\{y_1>0\}$ can be represented by $y_1=\rho^{-1}(y_2)$ as a function of $y_2$ for $y_2\in(0,p_2+\delta_0)$, where $\delta_0>0$ is a fixed small constant and $\rho^{-1}$ is the inverse of $\rho$. 

We can introduce a function $\psi=\psi_p$, which is independent of $y_1$, defined as
	\begin{equation*}\label{bdpsi}
		\psi(y_1,y_2) = \varphi(\rho^{-1}(y_2), y_2)  \quad \text{ in }\ \mathbb{R}\times (0,p_2+\delta_0).
	\end{equation*}
Note that $\psi(y)=\varphi(y)=u(y)$ for $y\in \p\Omega$ near $p$, and we will reduce the $C^{1,\alpha}$ estimate for $u$ in the $e_2$-direction to that of $\psi$. 
Similarly, by considering the other half of boundary $\p\Omega\cap\{y_1<0\}$, we can define a function $\psi_q$, which is independent of $y_1$. 

By subtracting $\ell_p$ from $u$, we may assume that the support $\ell_p = 0$, such that $u = 0$ on $\overline{pq}$ and $u \ge 0$ in $\Omega$.
It can be verified that  
	\begin{equation*}
		[u]_{C^{1,\beta/k}\{y\}} \le \max \{ [\psi_p]_{C^{1,\beta/k}\{p\}}, [\psi_q]_{C^{1,\beta/k}\{q\}}\} \ \ \text{ for any } y \in \overline{pq}.
	\end{equation*}
In the following we estimate $[\psi_p]_{C^{1,\beta/k}\{p\}}$. The same applies to $[\psi_q]_{C^{1,\beta/k}\{q\}}$. 
For simplicity we drop the subscript $p$ in $\psi_p$. 

Since $u\geq0$ and $\psi=\ell_p=0$ on the segment $\overline{pq}$, it suffices to show that  
	\begin{equation}\label{psi-0}
		\psi(y) \le C|y-p|^{1+\beta/k} \ \text{ for any } y \in \mathcal I := \{ p + s e_2 ~|~ -p_2 < s < \delta_0 \}. 
	\end{equation}
Notice that $\varphi$ satisfies $\varphi(p_1) = \varphi(q_1) =0$, and $\varphi \ge 0$.
Hence $\varphi'(p_1)=0$, $\varphi''(p_1) \ge 0$, and $\varphi'(q_1)=0$, $\varphi''(q_1) \ge 0$.
We divide the proof into two cases.

\noindent{\bf Case 1:}  $\beta \in (0,1]$.
We express $\varphi$ using its Taylor expansion at $p_1$ by
	\begin{equation}\label{taylor-00}
	\begin{split}
		\varphi(y_1) & = \sum_{i=2}^k\varphi^{(i)}(p_1)(y_1-p_1)^i + O(|y_1-p_1|^{k+\beta}).
	\end{split}
	\end{equation}
First, we estimate the derivatives $\varphi^{(i)}(p_1)$ as follows. 
Since $\varphi'(q_1)=\varphi'(p_1)=0$, we have
	\begin{equation}\label{int20}
		\int_{q_1}^{p_1} \varphi''(s)\,ds = \varphi'(p_1) - \varphi'(q_1) = 0.
	\end{equation}
Thus, there must exist a point $s_2\in(q_1,p_1)$ such that $\varphi''(s_2)=0$. 
Given that $\varphi''(q_1)\geq0$ and $\varphi''(p_1)\geq0$, from \eqref{int20} $\varphi''$ has a minimum point $s_3\in(q_1,p_1)$ where $\varphi^{(3)}(s_3)=0$. 
Unless $\varphi''\equiv0$ on $[q_1,p_1]$, one has $\varphi''(s_3)<0$. 
By Lemma \ref{lem-5.1}, we know that $\varphi^{(i)}(-z_1)=0$, $i=2,3,\dots,k-1$ in the $y$-coordinates.  	
Hence, we have $s_3\neq -z_1$. In the event that $\varphi''\equiv0$ on $[q_1,p_1]$, we can freely choose another $s_3$ to ensure that $s_3\neq -z_1$. 
	
Since $\varphi^{(3)}(s_3)=0$ and $\varphi^{(3)}(-z_1)=0$ at two distinct points $s_3$ and $-z_1$, there must exist a point $s_4$ in the interval between $s_3$ and $-z_1$ such that $\varphi^{(4)}(s_4)=0$ with $s_4\neq -z_1$. By induction, there exists a point $s_i$ in the interval between $s_{i-1}$ and $-z_1$ such that $s_i\neq -z_1$ and $\varphi^{(i)}(s_i)=0$ for $i=4,\cdots,k$. 
	
Denote the interval $\I=(-z_1,p_1)\cup(q_1,p_1)$. We get $s_i\in\I$ for all $i=2,\cdots,k$, and the length of $\I$ is bounded by $\bar a:= |p_1| + |z_1|$.
By \eqref{appy2}, we have
	\begin{equation}\label{bar-a}
	\begin{split}
		\bar a &= \left( \left(|p_1|+|z_1|\right)^{k-2} \right)^{\frac{1}{k-2}} \leq \left( C(|p_1|^{k-2}+|z_1|^{k-2}) \right)^{\frac{1}{k-2}} \\
			&\leq \frac{C\left(|p_1|^k+|z_1|^{k-2}p_1^2\right)^{\frac{1}{k-2}}}{p_1^{\frac{2}{k-2}}} \leq Cp_1^{-\frac{2}{k-2}}p_2^{\frac{1}{k-2}}.
	\end{split}
	\end{equation}
Hence, by the assumption $\varphi\in C^{k,\beta}$ we obtain
	\begin{equation*}\label{var-k-b}
		|\varphi^{(k)}(y_1)| \leq C|y_1-s_k|^\beta \leq C\bar a^\beta \ \ \text{ for any }\,y_1\in\I,
	\end{equation*}
and 
	\begin{equation*}\label{var-k-1-b} 
		|\varphi^{(k-1)}(y_1)| \leq \Big| \int_{s_{k-1}}^{y_1}\varphi^{(k)}(s) ds \Big|  \le C \bar a^{1+\beta} \ \ \text{ for any }\,y_1\in\I.
	\end{equation*}
Thus, for any $i=4,\cdots,k$, we get
	\begin{equation}\label{var-k-i-b} 
		|\varphi^{(i)}(y_1)| \leq  C \bar a^{k-i+\beta} \ \ \text{ for any }\,y_1\in\I.
	\end{equation}
For $i=2,3$, since $s_2, s_3\in(q_1,p_1)$, we have
	\begin{equation}\label{var-3-b}
		|\varphi^{(3)}(y_1)| \leq  \Big| \int_{s_3}^{y_1}\varphi^{(4)}(s) ds \Big| \leq C \bar a^{k-4+\beta}p_1 \ \ \text{ for any }\,y_1\in(q_1,p_1),
	\end{equation}
where the last inequality follows from $|y_1-s_3|\leq |q_1-p_1| \leq Cp_1$ due to $(i)$ of Lemma \ref{lem-boundary},
and thus
	\begin{equation}\label{var-2-b} 
		|\varphi''(y_1)| \le \Big| \int_{s_2}^{y_1}\varphi^{(3)}(s) ds \Big|  \le C \bar a^{k-4+\beta}p_1^2 \ \ \text{ for any }\,y_1\in(q_1,p_1).
	\end{equation}

Let $t=y_1-p_1$. Combining \eqref{taylor-00} and \eqref{var-k-i-b}--\eqref{var-2-b}, we obtain 
	\begin{equation} \label{etr}
		\varphi(y_1) \le C \big( \bar a^{k-4+\beta} p_1^2 t^2 + \bar a^{k-4+\beta} p_1 |t|^3 + \sum_{i=4}^k \bar a^{k-i+\beta} |t|^i + |t|^{k+\beta} \big).  
	\end{equation}
Let $\hat \varphi (y_1)$ denote the right-hand side of \eqref{etr}, which is monotone for $t > 0$.
Let $h= y_2 - p_2$. In the variables $(t,h)$, to prove \eqref{psi-0} it suffices to show that
	\begin{equation}\label{psi-1}
		\psi(h+p_2) \le C|h|^{1+\beta/k}\ \ \text{ for any }\, h \in (-p_2, \delta_0).
	\end{equation}
We will verify \eqref{psi-1} for $h$ in different intervals respectively. 
	
In the case $h\in(-p_2, 0]$, define a function $\hat\psi$ of $y_2$ such that
	\begin{equation*}
		\hat\psi(y_2) = \hat\varphi(y_1)\quad\text{ on the segment }\overline{op}=\{\theta p : \theta\in[0,1]\}.
	\end{equation*}
Since $\varphi\leq\hat\varphi$ in \eqref{etr} and $\hat\varphi$ is monotone, we have $\psi\leq\hat\psi$. 
On the segment $\overline{op}$, we have $t=\frac{p_1}{p_2}h$.
From $(i)$ of Lemma \ref{lem-boundary}, $p_1\leq Cp_2^{1/k}$. Since $|h|\leq p_2$, we then have
	\begin{equation}\label{esont}
		|t| \leq C\frac{|h|^{1-\frac1k}}{p_2^{1-\frac1k}}|h|^{\frac1k} \leq C|h|^{\frac1k}.
	\end{equation}
Hence, by the definitions of $\hat\varphi$ and $\hat\psi$, \eqref{esont}, and $|h|\leq p_2$, we get
	\begin{equation*} 
		\hat \psi(y_2)  \leq C \left( \bar a^{k-4+\beta}\frac{p_1^4}{p_2^2}|h|^2 + \sum_{i=4}^k \bar a^{k-i+\beta}\frac{p_1^i}{p_2^i}|h|^i + |h|^{1+\frac{\beta}{k}} \right).
	\end{equation*}
By \eqref{bar-a} and $p_1\leq Cp_2^{1/k}$, we have
	\begin{equation}\label{hat-psi}
	\begin{split}
		\hat \psi(y_2)  &\leq C \left( p_1^{\frac{2(k-\beta)}{k-2}}p_2^{ \frac{k-4+\beta}{k-2}}\left|\frac{h}{p_2}\right|^2 + \sum_{i=4}^k p_1^{\frac{(i-2)k-2\beta}{k-2}}p_2^{ \frac{k-i+\beta}{k-2}}\left| \frac{h}{p_2} \right|^i + |h|^{1+\frac{\beta}{k}} \right) \\
			&\leq C \left( p_2^{1+\frac{\beta}{k}}\left|\frac{h}{p_2}\right|^2 + \sum_{i=4}^k p_2^{1+\frac{\beta}{k}-i}\left| h \right|^i + |h|^{1+\frac{\beta}{k}} \right) \\
			&\leq C |h|^{1+\frac{\beta}{k}}  \left( \left|\frac{h}{p_2}\right|^{1-\frac{\beta}{k}} + \sum_{i=4}^k \left| \frac{h}{p_2} \right|^{i-1-\frac{\beta}{k}} + 1 \right) \leq C |h|^{1+\frac{\beta}{k}}, 
	\end{split}
	\end{equation}
where the last inequality follows from the fact that $|h|\leq p_2$. 
Hence, \eqref{psi-1} is proved for $h\in(-p_2,0]$.

In the case $h\in(0,\delta_0)$, let $t_i = 2^{i-1}p_1$, $h_i = \rho(t_i)$, and $P_i = (t_i, h_i)\in \p\Omega$ for $i\ge 1$.
We will prove \eqref{psi-1} for $y_2\in[h_i,h_{i+1})$, respectively. Define the segment
	\begin{equation*}
		L_i = \left\{ y\in \R^2 ~:~ y_2 = h_i + \rho'(t_i)(y_1 - t_i), \ h_i \le y_2 < h_{i+1} \right\},
	\end{equation*}
which is tangential to $\p\Omega$ at $P_i$. 
Let $\hat \psi_i$ be a function of $y_2$ such that
	\begin{equation*}
		\hat\psi_i(y_2) = \hat\varphi(y_1)\quad\text{ on the segment } L_i.
	\end{equation*}
Then $\psi \le \hat \psi_i$ on $L_i$, and it suffices to prove \eqref{psi-1} for $\hat\psi_i$. 

When $i=1$, the linear relation between $h$ and $t$ is given by
	\begin{equation*}
		h = y_2-h_1 = \rho'(p_1)(y_1-t_1) = \rho'(p_1)t.
	\end{equation*}
By $(i)$ and $(ii)$ of Lemma \ref{lem-boundary}, we get $\rho'(p_1) \approx p_2/p_1$, and hence
	\begin{equation*}
		t \leq C\frac{p_1}{p_2}h.
	\end{equation*}
By $(i)$ of Lemma \ref{lem-boundary}, we have $\rho(2p_1)\leq C\rho(p_1)$ for some constant $C>1$ that depends only on $k$. Consequently,
	\begin{equation} \label{hbip}
		h\leq \rho(2p_1) - \rho(p_1) \leq C\rho(p_1) = Cp_2. 
	\end{equation}
Hence, by the computations in \eqref{hat-psi}, we can obtain that $\hat \psi_1 \le C h^{1+\frac\beta4}$, and the estimate \eqref{psi-1} holds for $i=1$. 

By induction we assume that \eqref{psi-1} holds for ${i-1}$ and will show that it also holds for $i$, $i\geq2$. 
On the line segment $L_i$, we have the linear relation
	\begin{equation}\label{eehi}
	\begin{split}
		h & = \rho'(t_i)(y_1-t_i) + h_i - p_2 \\
			& = \rho'(t_i)(y_1-p_1) + \big( h_i - p_2 - \rho'(t_i)(t_i-p_1) \big)\\
			& =: \rho'(t_i) t + \theta.
	\end{split}
	\end{equation}
By $(i)$ and $(ii)$ of Lemma \ref{lem-boundary}, we have $t_i\rho'(t_i) \leq C \rho(t_i)$. Then by the convexity of $\rho$, it follows that
	\begin{equation*}
		\rho'(t_i) \leq C\frac{\rho(t_i)}{t_i} \leq C\frac{\rho(t_i)-\rho(t_1)}{t_i-t_1} = C\frac{h_i-p_2}{t_i-p_1},
	\end{equation*}
which implies that $\theta \geq -C(h_i-p_2) \geq -Ch$ since $h\geq h_i-p_2$. 
Hence, by \eqref{eehi}, we have
	\begin{equation*} 
		t \le C \frac{h}{\rho'(t_i)}
	\end{equation*}
for a different constant $C>0$. Again, by $(i)$ and $(ii)$ of Lemma \ref{lem-boundary}, $\rho'(t_i)\approx h_i/t_i$. Hence, analog to \eqref{esont}, we have
	\begin{equation}\label{eeth}
		t \le C \frac{t_i}{h_i}h \leq \frac{C h}{h_i^{1-\frac{1}{k}}} \leq Ch^{\frac{1}{k}},
	\end{equation}
where the last inequality follows from $h\leq h_{i+1} \leq Ch_i$, similar to \eqref{hbip}.

Substituting \eqref{eeth} in $\hat\varphi$ in the right-hand side of \eqref{etr} and observing that $p_1\leq t_i$, we obtain 
	\begin{equation*}
		\hat\psi_i(y_2) \leq C \left( \bar a^{k-4+\beta}\frac{t_i^4}{h_i^2}h^2 + \bar a^{k-4+\beta}\frac{t_i^4}{h_i^3}h^3 + \sum_{j=4}^k \bar a^{k-j+\beta}\frac{t_i^j}{h_i^j}h^j + h^{1+\frac{\beta}{k}} \right).
	\end{equation*}
Since $h\leq Ch_i$, the second term is controlled by the first term, allowing us to write   
	\begin{equation}\label{icont}
		\hat\psi_i(y_2) \leq C \left( \bar a^{k-4+\beta}\frac{t_i^4}{h_i^2}h^2 + \sum_{j=4}^k \bar a^{k-j+\beta}\frac{t_i^j}{h_i^j}h^j + h^{1+\frac{\beta}{k}} \right).
	\end{equation}
By $(i)$ of Lemma \ref{lem-boundary}, we have $\frac{h_1}{t_1^2} \approx \kappa(z) + t_1^{k-2} < \kappa(z) + t_i^{k-2} \approx \frac{h_i}{t_i^2}$. Thus, from \eqref{bar-a}, we obtain
	\begin{equation}\label{barai}
		\bar a \leq C\left(\frac{h_1}{t_1^2}\right)^{\frac{1}{k-2}} \leq C\left(\frac{h_i}{t_i^2}\right)^{\frac{1}{k-2}}
	\end{equation}
for a constant $C>0$ independent of $i$. 

Inserting \eqref{barai} in \eqref{icont} and noting that $t_i\leq Ch_i^{1/k}$, we can then deduce, similar to the computation in \eqref{hat-psi} (where $p_1, p_2$ are replaced by $t_i, h_i$, respectively), that 
	\begin{equation}\label{hat-psi-i}
	\begin{split}
		\hat \psi_i(y_2)  &\leq C \left( t_i^{\frac{2(k-\beta)}{k-2}}h_i^{\frac{k-4+\beta}{k-2}}\left(\frac{h}{h_i}\right)^2 + \sum_{j=4}^k t_i^{\frac{(j-2)k-2\beta}{k-2}}h_i^{\frac{k-j+\beta}{k-2}}\left( \frac{h}{h_i} \right)^j + h^{1+\frac{\beta}{k}} \right) \\
			&\leq C \left( h_i^{1+\frac{\beta}{k}}\left(\frac{h}{h_i}\right)^2 + \sum_{j=4}^k h_i^{1+\frac{\beta}{k}-j} h^j + h^{1+\frac{\beta}{k}} \right) \\
			&\leq C h^{1+\frac{\beta}{k}}  \left( \left(\frac{h}{h_i}\right)^{1-\frac{\beta}{k}} + \sum_{j=4}^k \left( \frac{h}{h_i} \right)^{j-1-\frac{\beta}{k}} + 1 \right) \leq C h^{1+\frac{\beta}{k}}, 
	\end{split}
	\end{equation}
where the last inequality follows from $h\leq Ch_i$ on the segment $L_i$. 
We have thus proved \eqref{psi-1}.  

{\bf Case 2:} \ $\beta \in (1,k]$.
It is important to note that the assumptions that $\p\Omega$ is $(k,\beta)$-order degenerate at $0$ (i.e. \eqref{k+b-deg1}) and that $\varphi$ satisfies condition ($\mathcal P2$) are only required at this point. 
In the $y$-coordinate, these assumptions imply that
	\begin{equation}\label{tcnee}
		\varphi^{(i)}(-z_1) = 0 \ \ \text{ for any } i,\ \  k+1\leq i\leq k+[\beta],
	\end{equation}
as noted in Remark \ref{re-5.1}.
Since $\varphi\in C^{k+\beta}$, we have the estimate
	\begin{equation*}\label{fb-kib} 
		|\varphi^{(k+[\beta])}(y_1)| \leq C\bar a^{\beta-[\beta]} \ \ \text{ for any }\,y_1\in\I, 
	\end{equation*}
where $\bar a$ is defined as in \eqref{bar-a} and $\I=(-z_1,p_1)\cup(q_1,p_1)$. By using \eqref{tcnee} and integration, we find
	\begin{equation*} 
		|\varphi^{(k+[\beta]-1)}(y_1)| \leq \Big| \int_{-z_1}^{y_1}\varphi^{(k+[\beta])}(s) ds \Big|  \le C \bar a^{1+\beta-[\beta]} \quad\text{for any } y_1\in\I. 
	\end{equation*}	
Thus, for all $i$ with $k+1\leq i\leq k+[\beta]$, we have
	\begin{equation} \label{b>1-1}
		|\varphi^{(i)}(y_1)| \leq C \bar a^{k-i+\beta} \ \ \text{ for any }\,y_1\in\I.
	\end{equation}		

For $4\leq i\leq k$, the estimate \eqref{var-k-i-b} holds. Hence, for any $i=4,\cdots,k+[\beta]$, we can write
	\begin{equation}\label{b>1-2} 
		|\varphi^{(i)}(y_1)| \leq C\bar a^{k-i+\beta} \ \ \text{ for any }\,y_1\in\I.
	\end{equation} 
For $i=2$ and $i=3$, we have the same estimates as in \eqref{var-2-b} and \eqref{var-3-b}, respectively. 

Consequently, similar to \eqref{etr}, the Taylor expansion of $\varphi$ at $p_1$ yields 
	\begin{equation*}\label{etr2}
		\varphi(y_1) \le C \big( \bar a^{k-4+\beta}|p_1|^2t^2 + \bar a^{k-4+\beta}|p_1| |t|^3 + \sum_{i=4}^{k+[\beta]} \bar a^{k-i+\beta} |t|^i + |t|^{k+\beta} \big).	
	\end{equation*}   
To verify \eqref{psi-1}, we first consider $h\in(-p_2,0]$, as in Case 1. 
Define $\overline{op}, \hat\varphi, \hat\psi$ as before. 
Similar to \eqref{hat-psi}, we have
	\begin{equation*}\label{hp2}
		\hat\psi(y_2) = \hat\varphi(y_1) \leq C |h|^{1+\frac{\beta}{k}}  \left( \left|\frac{h}{p_2}\right|^{1-\frac{\beta}{k}} + \sum_{i=4}^{k+[\beta]} \left| \frac{h}{p_2} \right|^{i-1-\frac{\beta}{k}} + 1 \right) \leq C |h|^{1+\frac{\beta}{k}},
	\end{equation*}
since $|h| \leq p_2$. Therefore, \eqref{psi-1} holds for $h\in(-p_2,0]$.

The verification of \eqref{psi-1} for $h\in[0,\delta_0)$ is similar to that of Case 1, specifically by considering the intervals $[h_i,h_{i+1}]$ for $i=1,2,\cdots$.
The key components are \eqref{eeth} and \eqref{barai}. Following a similar computation to that in \eqref{hat-psi-i}, we can verify \eqref{psi-1}. Since the argument is analogous to that in Case 1, we omit the details here. 
\end{proof}

\begin{remark} 
\emph{
In the proof of Lemma \ref{lem-6.1}, we see that the estimates for higher-order derivatives \eqref{var-k-i-b} are derived differently from those for the second and third-order derivatives \eqref{var-3-b} and \eqref{var-2-b}.
Notably, if the domain $\Omega$ is uniformly convex, namely $k=2$, assumptions $(\mathcal P1)$--$(\mathcal P2)$ (as well as $(\mathcal P3)$ in Section \ref{S6}) are not required. 
For details, we refer readers to \cite{CTW-2022}. 
}
\end{remark}

\vspace{5pt}
\section{Convex Envelop} 
\label{S6}

Recall that the convex envelope $u$ is given in \eqref{u-ce}. Under the hypotheses of Theorem \ref{mt-3}, we show that the proof in Section \ref{S5} applies to $u$ and yields the global $C^{1,\beta/k}$ regularity, where $k \ge 4$ is an even integer and $\beta\in(0,k]$.

The global $C^{1,\alpha}$ for some $\alpha\in(0,1]$ was obtained in \cite{CTW-2022,DF-2015} under the condition that $\p\Omega$ is uniformly convex. In Theorem \ref{mt-3}, $\Omega$ is not uniformly convex. 
By the proof of Lemma \ref{thm-gra-est} and the modification \eqref{neww} for the general case, we conclude that at any point $x_0\in\overline\Omega$, there is a unique support plane of $u$ at $x_0$. 
By definition, we have $u\leq w$ in $\Omega$. 
Since $\Omega$ is strictly convex, we have $u=\varphi$ on $\partial\Omega$, where $\varphi=w|_{\partial\Omega}$.  

\begin{proof}[Proof of Theorem \ref{mt-3}]
Denote $G=\{x\in\overline\Omega ~|~ u(x)<w(x)\}$. 
At $x_0\in\Omega\setminus G$, since $u(x_0)=w(x_0)$ and $u\leq w$, we have 
	\begin{equation}\label{inpts}
		[u]_{C^{1,\beta/k}\{x_0\}} \le [w]_{C^{1,\beta/k}\{x_0\}}. 
	\end{equation}
Hence, to prove Theorem \ref{mt-3}, it suffices to prove 
$[u]_{C^{1,\beta/k}\{x_0\}}\leq C_0$ for $x_0\in G$. 

In fact, if $x_0\in\p\Omega$, we can choose a sequence $\{x_i\}\subset\Omega$ such that $x_i\to x_0$.
Then, by the convexity of $u$ and the uniform estimate $[u]_{C^{1,\beta/k}\{x_i\}}\leq C_0$, it follows that $[u]_{C^{1,\beta/k}\{x_0\}}\leq C_0$. 

For any point $x_0\in G$, by \eqref{u-ce} the contact set $\mathcal{C}_0$ at $x_0$, defined in \eqref{ctset}, must be a convex set with more than one point. 
Let $x_0=\sum_{i=1}^ma_ip_i$, where $a_i>0, \sum a_i=1$, and $p_i\in\overline\Omega\setminus G$ are extreme points of $\mathcal{C}_0$. 
If $[u]_{C^{1,\beta/k}\{p_i\}}\leq C_0$ for all 
$i=1,\cdots,m \, (\leq 3)$, then by the convexity of $u$, we have $[u]_{C^{1,\beta/k}\{x_0\}}\leq C_0$.
 
Therefore, the proof reduces to proving $[u]_{C^{1,\beta/k}\{p_i\}}\leq C_0$ for all $i=1,\cdots,m$. 
Consider an arbitrary segment $\overline{p_ip_j}$. For simplicity, we denote this segment by $\overline{pq}$, where $p, q\in\overline\Omega\setminus G$ and $u$ is linear on $\overline{pq}$. 

If both $p, q\in\Omega$, then from \eqref{inpts}, we have $[u]_{C^{1,\beta/k}\{p\}} \le [w]_{C^{1,\beta/k}\{p\}}$ and $[u]_{C^{1,\beta/k}\{q\}} \le [w]_{C^{1,\beta/k}\{q\}}$, concluding the proof. If both $p, q\in\partial\Omega$, we are in the same situation as in Lemma \ref{lem-6.1}. Then, using the estimate \eqref{6.5}, we complete the proof. 

Hence, we just need to consider the case when $p\in\partial\Omega$ and $q\in\Omega$. In this case, it suffices to obtain $[u]_{C^{1,\beta/k}\{p\}} \le C_0$.
As in the proof of Theorem \ref{thm-2.1}, we may assume that $\p\Omega$ is locally given by \eqref{x2-rho-1} and \eqref{x2-rho-2}, with $p\in\p\Omega\cap\{x_2<\delta_1\}$ being close to the origin.
By subtracting the support of $w$ at $0$, we can assume 
	\begin{equation}\label{dew0}
		w\geq w(0)=0\ \ \text{ and } Dw(0)=0.
	\end{equation} 
 
It is also worth noting that we may assume 
$|p-q|<\delta_3$ for a very small constant $\delta_3>0$. Otherwise, by case $(i)$ in the proof of Theorem \ref{thm-2.1}, the estimate $[u]_{C^{1,\beta/k}\{p\}} \le C_0$ follows from the proof of Theorem \ref{lem-3.2}.
Without loss of generality, we may assume that $p_1>0$. 

The proof follows the approach of Lemma \ref{lem-6.1}, with some modifications for obtaining the derivative estimates. 
First, let $z\in\p\Omega$ near the origin such that the tangential vector $\tau_z$ of $\p\Omega$ is parallel to $\overline{pq}$. 
We may assume $z_1\geq0$ in the following. If $z_1\leq0$, the proof proceeds in a similar manner.
After applying the coordinate change $\mathcal{T}$ from \eqref{t-k}, we have $\mathcal{T}(z)=0$, and $\overline{pq}$ is parallel to the $e_1$-axis in the new $y$-coordinates. By Lemma \ref{lem-boundary}, we have $|q_1|\leq C|p_1|$. 

Similar to \eqref{bar-a}, we define $\bar a:=|p_1|+|z_1|$. 
By \eqref{dew0} and condition $(\mathcal P3)$, since $\mathcal{T}$ is an affine transform, in the $y$-coordinates, we have
	\begin{equation}\label{dew1}
		D^{(i)}w(-z) = 0 \quad \text{ for }\ i=1,\cdots,k+[\beta].
	\end{equation}
Since $w\in C^{k+\beta}(\overline\Omega)$, by \eqref{dew1} and successive integration, we obtain 
	\begin{equation}\label{dew2}
		|D^{(i)}w(p)|, |D^{(i)}w(q)| \leq C\bar a^{k-i+\beta} \quad \text{ for }\ i=1,\cdots,k+[\beta].
	\end{equation}

We define the functions $\bar u, \bar w$, and $\bar\varphi$ by subtracting $\ell_q=a_0+a_1y_1+a_2y_2$, the support of $w$ at $q$, from $u, w, \varphi$, respectively, i.e., 
	\begin{equation}\label{subou} 
		\bar u= u -\ell_q;\quad \bar w=w -\ell_q;\quad \bar\varphi=\bar w|_{\p\Omega}.
	\end{equation}
By \eqref{dew2}, we have $|a_1|, |a_2| \leq C\bar a^{k-1+\beta}$. 
Consequently, $\bar u=0$ on $\overline{pq}$, and $\bar u\geq0$ in $\Omega$. 
From \eqref{subou}, we have
	\begin{equation}\label{dew22}
		\bar\varphi(y_1) = \bar w(y_1,\rho(y_1)) = w(y_1,\rho(y_1)) - (a_0+a_1y_1+a_2\rho(y_1)),
	\end{equation}
where $\rho$ is the boundary function in $y$-coordinates, denoted as $\tilde\rho$ in Lemma \ref{lem-boundary}.
By \eqref{dew1} and \eqref{dew2}, we can verify that for $i=1,\cdots,k+[\beta]$, 
	\begin{align}\label{dew3} 
		& |\bar\varphi^{(i)}(y_1)| \leq C\bar a^{k-i+\beta} \ \ \text{ for any }\, y_1\in\I,
	\end{align}
where $\I:=[-z_1,p_1]\cup[q_1,p_1]$
if $q_1<p_1$; $\I:=[-z_1,q_1]$ if $q_1>p_1$.
This is analogous to \eqref{var-k-i-b} and \eqref{b>1-2} in the proof of Lemma \ref{lem-6.1}. 

To derive the analogue of \eqref{var-3-b} and \eqref{var-2-b}, we need to make the following modification, as $q\notin\p\Omega$. 
Let $\eta=\eta(y_1)$ be a smooth function defined in a neighbourhood of $[q_1,p_1]$ (or $[p_1,q_1]$) such that
$\eta(p_1)=\eta(q_1)=p_2$ and 
	\begin{align}\label{dew4}
		\rho(y_1) \leq \eta(y_1) \leq p_2 &\ \ \text{ for any }\, y_1 \in [q_1,p_1] \text{ (or $[p_1,q_1]$)};\\
		\eta^{(i)}(p_1) = \rho^{(i)}(p_1) &\ \ \text{ for any }\, i=0,1,\cdots,k+[\beta].  \nonumber
	\end{align}
Define $\check\varphi(y_1)=\bar w(y_1,\eta(y_1))$, a function independent of $y_2$. This function $\check\varphi$ satisfies $\check\varphi'(p_1)=0$, $\check\varphi''(p_1) \ge 0$; and $\check\varphi'(q_1)=0$, $\check\varphi''(q_1) \ge 0$.
By the same argument following \eqref{int20}, there exist two points $s_2, s_3\in(q_1,p_1)$ (or $(p_1,q_1)$) such that 
	\begin{equation}\label{dew44}
		\check\varphi''(s_2)=0,\quad \check\varphi^{(3)}(s_3)=0.
	\end{equation} 

For any $y_1\in[q_1,p_1]$ (or $[p_1,q_1]$), we denote $y^\eta:=(y_1,\eta(y_1))$.  
Since $w\in C^{k+\beta}(\overline\Omega)$, by \eqref{dew1} and \eqref{dew4}, similar to \eqref{dew2}, we get
	\begin{equation}\label{dew5}
		 |D^{(i)}w(y^\eta)| \leq C\bar a^{k-i+\beta} \quad \text{ for }\ i=1,\cdots,k+[\beta].
	\end{equation}
After subtracting $\ell_q$ and restricting on the curve $\{y_2=\eta(y_1)\}$, similar to \eqref{dew22}, we have
	\begin{equation*} 
		\check\varphi(y_1) = \bar w(y_1,\eta(y_1)) = w(y_1,\eta(y_1)) - (a_0+a_1y_1+a_2\eta(y_1)).
	\end{equation*}
Hence, by differentiation and \eqref{dew5}, we can obtain the estimate \eqref{dew3} for the function $\check\varphi$. 
In particular, for $i=4$, we find
	\begin{equation}\label{dew6}
		|\check\varphi^{(4)}(y_1)| \leq C\bar a^{k-4+\beta}\ \ \text{ for any }\, y_1 \in [q_1,p_1] \text{ (or $[p_1,q_1]$)}. 
	\end{equation} 
Combining \eqref{dew44} and \eqref{dew6}, we then have
	\begin{equation*} 
		|\check\varphi^{(3)}(p_1)| \leq C\bar a^{k-4+\beta}|p_1| \quad\text{and} \quad |\check\varphi''(p_1)| \leq C\bar a^{k-4+\beta} p_1^2. 
	\end{equation*}	
From the construction \eqref{dew4}, we know that $\check\varphi^{(3)}(p_1)=\bar\varphi^{(3)}(p_1)$ and $\check\varphi''(p_1)=\bar\varphi''(p_1)$. Therefore, we obtain the desired estimates, analogous to \eqref{var-3-b} and \eqref{var-2-b}, that 
	\begin{equation*}\label{dew7} 
		|\bar\varphi^{(3)}(p_1)| \leq C\bar a^{k-4+\beta}|p_1| \quad\text{and} \quad |\bar\varphi''(p_1)| \leq C\bar a^{k-4+\beta} p_1^2. 
	\end{equation*}	
Once having the above derivative estimates at the point $p\in\p\Omega$, we can then plug them into the Taylor expansion \eqref{taylor-00}. 
	
By substituting the function $\bar\varphi$ for $\varphi$ in Lemma \ref{lem-6.1},
the proof can then be completed by following the argument line by line as in Lemma \ref{lem-6.1}.
\end{proof}

\vskip5pt 
\appendix
\section{Some estimates on polynomials} 

\begin{lemma}\label{p-k}
Let $P_k(t)=(t+1)^k-t^k-kt^{k-1}$, where $k\ge 2$ is an even integer.
Then,  
	\begin{equation}\label{p-ct}
		C_k\big(t^{k-2}+1 \big) \leq P_k(t) \leq C_k'\big(t^{k-2}+1 \big), \quad t\in(-\infty, +\infty),
	\end{equation}
for some constants $C_k, C_k'>0$ depending only on $k$.
\end{lemma}
\begin{proof}
The second inequality follows from Young's inequality, so we only need to prove the first inequality. 
It is easy to see that \eqref{p-ct} holds for $k=2$. 

Assume that for $k\geq2$, there is a $C_k>0$ such that
	\begin{equation*}\label{p-ck}
		P_k(t)  \ge C_k (t^{k-2}+1), \quad t\in(-\infty, +\infty).
	\end{equation*}
By computation, we find that
	\begin{equation*}\label{p-ck2}
	\begin{split}
		P_{k+2}(t) &=(t+1)^{k+2}-t^{k+2}-(k+2)t^{k+1} \\
			& = t^2 P_k(t) + (2t+1)(t+1)^k- 2t^{k+1} \\
			& \ge C_k t^2 (t^{k-2}+1) + (2t+1)(t+1)^k- 2t^{k+1}. 
	\end{split}
	\end{equation*}
Since $k$ is even,
	\begin{equation*}\label{p+q} 
		(2t+1)(t+1)^k- 2t^{k+1} - t^k = (2t+1)\Big((t+1)^k - t^k \Big) \ge 0. 
	\end{equation*}
This implies that
	\begin{equation*}
	\begin{split}
		P_{k+2}(t) &\ge C_k t^2 (t^{k-2}+1) + t^k \\
				&= (C_k+1)t^k + C_kt^2.
	\end{split}
	\end{equation*}

For $|t| \ge \frac{1}{2(k+2)}$, we can derive
	\begin{equation}\label{p-ck4} 
	\begin{split}
		P_{k+2}(t) & \ge  (C_k +1)t^k + C_k \left( \frac{1}{2(k+2)}\right)^2 \\
			& \ge \frac{C_k}{4(k+2)^2} \big(t^k +1\big). 
	\end{split}
	\end{equation}
For $|t| < \frac{1}{2(k+2)}$, letting $\binom{a}{b}$ be the binomial coefficient, we have
	\begin{equation*} 
	\begin{split}
		P_{k+2}(t) & = \sum_{p=0, \text{even}}^{k-2} \Big( \binom{k+2}{p} t^p + \binom{k+2}{p+1} t^{p+1} \Big) + \frac{(k+2)!}{k! \cdot 2!}t^{k} \\
			& =  \sum_{p=0, \text{even}}^{k-2}\binom{k+2}{p} t^p  \left( 1 + \frac{k+2-p}{p+1} t \right) + \frac12(k+2)(k+1)t^{k}.
	\end{split}
	\end{equation*}
Note that $k$ is even and the coefficient $\frac{k+2-p}{p+1}$ is monotone decreasing in $p$. Thus, for $|t| < \frac{1}{2(k+2)}$, we have
	\begin{equation}\label{p-ck3}
	\begin{split}
		P_{k+2}(t) & > \big(1+(k+2)t\big) + \frac12 t^k \\ 
			& > \frac12 + \frac12t^k = \frac12(t^k+1). 
	\end{split}
	\end{equation}

Therefore, combining \eqref{p-ck4} and \eqref{p-ck3}, we can obtain the first inequality of \eqref{p-ct} for $P_{k+2}$, where 
$ C_{k+2}= \min \{\frac14, \frac{C_k}{4(k+2)^2}\}.$
\end{proof}

\begin{lemma}\label{q-k}
Let $Q_k(t)=t^k + k(t+1)^{k-1}- (t+1)^{k}$, where $k\ge 2$ is an even integer.
Then, 
	\begin{equation*}
		\bar C_k\big(t^{k-2}+1 \big) \leq Q_k(t) \leq \bar C_k' \big(t^{k-2}+1 \big) , \quad t\in(-\infty, +\infty),
	\end{equation*}
for some positive constants $\bar C_k, \bar C_k'$ depending only on $k$.
\end{lemma}

\begin{proof}
Taking $s:=-(t+1)$, we get $Q_k(t)=(s+1)^k - ks^{k-1} - s^k = P_k(s)$.
Then, by Lemma \ref{p-k}, we have
    \begin{equation*}
	    Q_k(t) = P_k(s) \le C'_k\big(s^{k-2} +1\big) = C'_k\big((t+1)^{k-2} +1\big) \le C \big(t^{k-2} + 1\big)
    \end{equation*}
and
    \begin{equation*}
	    Q_k(t) = P_k(s) \ge C_k\big(s^{k-2} +1\big) = C_k\big((t+1)^{k-2} +1\big).
    \end{equation*}
If $|t|>2$, then 
	$$Q_k(t) \ge C_k \big((\frac{t}{2})^{k-2} +1 \big) \ge C \big(t^{k-2} +1 \big).$$ 
If $|t| \le 2$, we have 
	$$Q_k(t) \ge C_k \ge \frac{C_k}{2} \big((\frac{t}{2})^{k-2} +1 \big) \ge C \big(t^{k-2} +1 \big).$$
Thus, the desired result is proved. 
\end{proof}

\end{document}